\documentclass[11pt,a4]{article}
\usepackage[colorlinks=true, pdfstartview=FitV, linkcolor=blue, citecolor=blue,urlcolor=blue]{hyperref}
\usepackage[applemac]{inputenc}
\usepackage{amssymb,amsmath,amsthm,amsfonts,amscd,amstext}
\usepackage{mathtools}
\usepackage[mathscr]{eucal}
\usepackage{multirow}
\usepackage{xypic}
\usepackage{array}
\usepackage{graphicx}
\newtheorem{thm}{Theorem}[section]
\newtheorem{lem}[thm]{Lemma}
\newtheorem{prop}[thm]{Proposition}
\newtheorem{dfthm}[thm]{Definition / Theorem}
\newtheorem{cor}[thm]{Corollary}
\newtheorem{df}[thm]{Definition\rm}
\newtheorem{rem}{\it Remark\/}
\def\C{{\mathbb C}}  

\def\deg  {{\rm deg}}

\def\E{{\mathbb E}}    

\def\H {\mathbb H} 

\def\Hom{\operatorname{Hom}} 
\def\id{\operatorname{id}}  

\def\ker{\operatorname{ker}}  

\def\M {\mathbb M}   

\def\pr{\operatorname{pr}} 

\def\Q{{\mathbb Q}}   
\def\R{{\mathbb R}}    

\def\W{{\mathbb W}}    
\def\Z{{\mathbb Z}}    
\def\N{{\mathbb N}}    

\def\supp{\operatorname{supp}}  

\def\og{\leavevmode\raise.3ex\hbox{$\scriptscriptstyle\langle\!\langle$~}}

\def\fg{\leavevmode\raise.3ex\hbox{~$\!\scriptscriptstyle\,\rangle\!\rangle$}}
\begin{document}
\title{\bf (Co)monads in Free Probability Theory}
\author
{Roland Friedrich\thanks{Saarland University, friedrich@math.uni-sb.de. Supported by the  ERC advanced grant ``Noncommutative distributions in free probability". 
}}

\maketitle
\date{}
\begin{abstract}
We discuss free probability theory and free harmonic analysis from a categorical perspective. In order to do so, we extend first the set of analytic convolutions and operations and then show that the comonadic structure governing free probability is isomorphic to several well-known categories of algebras, such as, e.g., Witt vectors, differential algebras, etc. Within this framework moment-cumulant formulae are shown to correspond to natural transformations and not to be exclusive to probability theories. Finally, we start to discuss free probability and in particular free harmonic analysis from the point of view of algebraic theories and operads.
\end{abstract}

\tableofcontents
\section{Introduction}
If one is dealing with maps which are incompatible with a given algebraic structure such as, e.g. linear vs. multiplicative, ``a map" vs. a linear one or more general situations, there is, a priori, no information how to express the value of these maps on combined quantities from the values on the constituents themselves. In the simplest case this means, e.g.:
\begin{eqnarray*}
f(x+y)&=&?\\
f(x\cdot y)&=&?
\end{eqnarray*}
However, it happens that some kind of ``independence" is present which permits one to express the value on combined quantities in terms of its base constituents and the values on them. This situation is common in probability theory, both commutative and non-commutative, but it appears in other contexts as-well. Let us consider Joyal's $\delta$-rings~\cite{Joy1} as a prototypical example. Let $p$ be a fixed prime number and $q:=p^n$, $n\in\N^*$.
\begin{df}[\cite{Joy1}]
Let $A\in\mathbf{cAlg}_k$. A {\bf $\delta$-ring} is a pair $(A,\delta)$ where $\delta:A\rightarrow A$ is a map which satisfies the following identities:
\begin{eqnarray*}
\delta(1)&=&0,\\
\delta(x+y)&=&\delta(x)+\delta(y)-\sum_{i=1}^{q-1}\frac{1}{p} {{q}\choose{i}}x^iy^{q-i},\\
\delta(x\cdot y)&=&x^q\delta(y)+\delta(x)y^q+p\delta(x)\delta(y).
\end{eqnarray*}
\end{df} 
Further, if $q=p$ then the corresponding $\delta$-ring is isomorphic to the ring of Witt vectors on $A$ as has been shown in~\cite{Joy1}.

In~\cite{FMcK2011,FMcK2012} J.~McKay and the author, established a surprising connection of free probability theory and free harmonic analysis with the ring of Witt vectors and complex cobordism, motivated by the relations J.~McKay suggested Free Probability should have with Monstrous Moonshine~\cite{FrdMcK}.

Free harmonic analysis originates in the work of D.V. Voiculescu~\cite{V1985,V1986,VDN} and it develops a framework which parallels the theory in classical probability, however with the notion of independence replaced by freeness. Subsequently,  H. Berkovici and D.V. Voiculescu~\cite{BV} further developed the underlying complex analytic theory and then H. Berkovici and V. Pata~\cite{BPB} established with the so-called {\em Berkovici-Pata bijection}~\cite{BPB} the connection, both at the algebraic and analytic level,  with classically infinitely divisible probability measures. R. Speicher and A. Nica~\cite{NS} on the other hand treated infinitely divisible distributions and limit theorems with their methods based on free cumulants and cumulants of operators on Fock spaces. 

The link between the theory of Witt vectors and free probability theory~\cite{FMcK2012} led to a formula, a logarithm, which linearises Voiculescu's celebrated $S$-transform~\cite{V1987} but also to the introduction of an additional convolution which is, a priori,  
neither given by the addition nor the multiplication of the underlying algebra of random variables. Notably, in~\cite{FMcK2013a} it was shown that in higher dimensions the free multiplicative convolution can not be linearised and the algebraic groups related to it determined, as pro-unipotent or Borel groups. 

Previously and independently, M. Mastnak and A. Nica~\cite{MN} showed, by using Lie theoretic methods, that the one-dimensional $S$-transform can be linearised. The connection between the additive and multiplicative free convolution which was obtained at the algebraic level in~\cite{MN,FMcK2012} was subsequently studied with analytic methods by G. Cébron~\cite{C2014}. He  investigated a homomorphism of exponential type between freely infinitely divisible distributions and also related it to a matricial model he introduced. 

M. Anshelevich and O. Arizmendi~\cite{AA2016} then considered general exponential homomorphisms in non-commutative probability and discussed the relation their map has with the versions in~\cite{MN,FMcK2012,C2014}.

Here we continue the investigation Witt vectors have with free harmonic analysis both at the algebraic and analytic level. In particular we considerably extend our first analytic results in~\cite{FMcK2013a} to the whole set of freely infinitely divisible probability measure with compact support. 

The content of the paper is summarised in the Table of Contents.

\section{Algebraic Aspects of Free Probability}
In this section we discuss the algebraic properties of the groups, which originally appeared in a series of articles by Voiculescu~\cite{V1985,V1986,V1987} as convolution semi-groups for the addition and multiplication of free random variables. Their analogs (and generalisations) have been obtained by combinatorial methods by Nica and Speicher~\cite{NS}, based on the notion of free cumulants, as introduced by Speicher~\cite{S1997}. 
General background information on free probability can be obtained from the monographs~\cite{NS,VDN}. Let us point out, that here we focus on the one-dimensional case.
\subsection{The category of algebraic probability spaces and freeness}
Let $k$ be a field of characteristic zero, and $\mathbf{Alg}_k$ the category of associative $k$-algebras, not necessarily unital, and $\mathbf{cAlg}_k$ the associative and commutative unital $k$-algebras. For $A\in\mathbf{cAlg}_k$, we denote by $A^{\times}$ its group of units, i.e. the set of invertible elements with respect to the multiplication. Next, we introduce the category of algebraic $k$-probability spaces. 
\begin{df}
Let $A\in\mathbf{Alg}_k$, with unit $1_A$, and let $\phi:A\rightarrow k$ be a pointed $k$-linear functional with $\phi(1_A)=1_k$. The pair $(A,\phi)$, is called a {\bf non-commutative $k$-probability space}, and the elements $a\in A$ the {\bf random variables}.
\end{df}
The algebraic probability spaces over $k$ form a category, which we denote by $\mathbf{Alg_kP}$. The subcategory of commutative algebraic probability spaces is denoted by $\mathbf{cAlg_kP}$. The full subcategory of $C^*$-probability spaces $\mathbf{C^*Alg_{\C}P}$, has a objects pairs $(A,\varphi)$, where $A$ is a $C^*$-algebra and $\phi$ a state.
\begin{df}
The {\bf law} or {\bf distribution} of a random variable $a\in A$, is the $k$-linear functional $\mu_a\in k[[x]]$,
given by $\mu_a(x^n):=\phi(a^n)$, $n\in\N$.
The coefficients $(m_n(a))_{n\in\N}$ of the power series $\mu_a$ are the {\bf moments} of $a$.
\end{df}

\begin{rem}
Let us briefly discuss a kind of ``Taylor approximation" for random variables. 
For $\mathfrak{m}:=A\setminus\{k1_A\}$ consider the decreasing filtration
$$ 
A\supset\mathfrak{m}\supset\mathfrak{m}^2\supset\mathfrak{m}^3\supset\dots.
$$
of subsets (ideals). Then, by restricting $\phi$ onto $\mathfrak{m}^n$, $n\in\N^*$, and, in particular, to the purely diagonal sets, we obtain higher order, localised information. \end{rem}

\begin{df}[Freeness]
Let $(A,\phi)\in\mathbf{Alg_kP}$ and $I$ an index set.  A family of sub-algebras $(A_i)_{i\in I}$, with $1_A\in A_i\subset A$, $i\in I$, is called {\bf freely independent} if for all $n\in\N^*$, and $a_k\in {A_{i_k}\cap\ker(\phi)}$, $1\leq k\leq n$, also
$$
a_1\cdots a_n\in{\ker(\phi)},
$$
holds, given that $i_k\neq i_{k+1}$, $1\leq k< n$.
\end{df}
One should note that the requirement for freeness is that {\em consecutive indices} must be distinct, thus $i_j=i_{j+2}$ is possible.

Let $I_{\phi}\subset\ker(\phi)$ be the largest two-sided ideal contained in the kernel of $\phi$, which is a linear subspace of $A$. Then all the elements of $I_{\phi}$ are free but have trivial moments. Hence, the relevant object for free probability is
the linear quotient space $\ker(\phi)/I_{\phi}$. 

The comparison with classical probability is illustrative.
\begin{df}[Independence]
Let $(A,\phi)\in\mathbf{cAlg_kP}$ be a commutative algebraic probability space, and $I$ an index set.  A family of sub-algebras $(A_i)_{i\in I}$, with $1_A\in A_i\subset A$, $i\in I$, is called {\bf independent} if for all $n\in\N^*$, and $a_k\in A_{i_k}$, $1\leq k\leq n$, with $A_{i_k}\neq A_{i_{k'}}$ for $k\neq k'$ (i.e. for mutually different sub-algebras )
$$
\phi(a_1\cdots a_n)=\prod_{i=1}^n \phi(a_i),
$$
holds.
\end{df}

In classical probability, independence of random variables corresponds to those elements for which the expectation, a linear functional, behaves like a character. 
In free probability those elements of the kernel are of importance which behave like elements of an algebra. 

\subsection{The affine group schemes related to free additive and multiplicative convolution}
Voiculescu~\cite{V1985,V1986,V1987} proved that for a pair of free random variables $\{a,b\}$ in a non-commutative probability space, with distributions $\mu_a$ and $\mu_b$, respectively, $\mu_{a+b}$ and $\mu_{ab}$ depend only on $\mu_a$ and $\mu_b$. 
\begin{itemize}
\item The distribution of $\mu_{a+b}$ is called the {\bf additively free convolution} of $\mu_a$ and $\mu_b$, and it is denoted by $\mu_a\boxplus\mu_b$.
\item The distribution of $\mu_{ab}$ is called the {\bf multiplicatively free
convolution} of $\mu_a$ and $\mu_b$, and it is denoted by $\mu_a\boxtimes\mu_b$.
\end{itemize}

In fact, the binary operations $\boxplus$ and $\boxtimes$, as induced by the sum and the product of free random variables, can be expressed by {\bf universal polynomials with integer coefficients}, and therefore these operations can be further investigated by algebraic methods.

Let us consider the following $k$-functors from $\mathbf{cAlg}_k$ to $\mathbf{Set}$, 
\begin{eqnarray*}
\mathbb{A}^{\N^*}(A) & := & \{(a_1,a_2, a_3,\dots)~|~a_i\in A, i\geq 1\},\\
\mathfrak{G}(A) & := & \{(a_1,a_2, a_3,\dots)~|~a_1\in A^{\times},  a_i\in A, i\geq 2\},\\
\mathfrak{G}_+(A) & := & \{(1_A,a_2,a_3,\dots)~|~a_i\in A, i\geq 2\},
\end{eqnarray*}
and their finite-dimensional projections for $n\in\N^*$,
\begin{eqnarray*}
\mathbb{A}^n(A) & := & \{(a_1,\dots,a_n)~|~a_i\in A, 1\leq i\leq n\},\\
\mathfrak{G}_n(A) & := & \{(a_1,\dots, a_n)~|~a_1\in A^{\times}, a_i\in A, 2\leq i\leq n\},\\
(\mathfrak{G}_+)_n(A) & := & \big\{(1_A,a_2,\dots,a_n)~|~a_i\in A, 2\leq i\leq n\}.
\end{eqnarray*} 
We note, that $\mathbb{A}^{\N^*}(A)$, i.e. the affine $\N^*$-space over $k$, $\mathfrak{G}_+(A)$ and $(\mathfrak{G})_n(A)$ are affine varieties. 

Let us restate the following results of Voiculescu in the language of formal groups.
\begin{prop}[\cite{V1985, V1986}]
\label{PropVRtrafo}
The additive free convolution $\boxplus$, defines an infinite {\bf commutative formal group law} over $\Z$ which is homogeneous, i.e.  there exists a countable sequence of polynomials $P_{\boxplus}=P=(P_1,P_2,P_3,\dots)$, $P_i\in\Z[x_1,\dots,x_n,y_1,\dots,y_n]$ such that
\begin{itemize}
\item $P_1(x_1,y_1)=x_1+ y_1,$
\item $P_n(x,y)=P_n(y,x)$, for all $n\in\N^*$, (commutative, symmetric)
\item $P_n(x_1,\dots, x_n, y_1,\dots, y_n)=x_n+y_n+\tilde{P}_n(x_1,\dots,x_{n-1},y_1,\dots,y_{n-1})$, (homogeneous of degree $n$ if $\operatorname{deg}(x_i):=\operatorname{deg}(y_i):=i$)
\item $\underline{0}=(0_n)_{n\in\N^*}$ is the neutral element.
\end{itemize}
\end{prop}

\begin{prop}[\cite{V1987}]
\label{PropVStrafo}
The multiplicative free convolution $\boxtimes$, defines an infinite {\bf commutative algebraic group law} over $\Z$,  i.e. there exists a countable sequence of polynomials   $Q_{\boxtimes}=Q=(Q_1,Q_2,Q_3,\dots)$, $Q_n\in\Z[x_1,\dots,x_n,y_1,\dots, y_n]$ such that: 
\begin{itemize}
\item $Q_1(x_1,y_1)=x_1\cdot y_1$
\item $Q_n(x,y)=Q_n(y,x)$ for all $n\in\N^*$, (commutative, symmetric)
\item $Q_n(x_1,\dots, x_n, y_1,\dots, y_n)=x_ny_1^n+x_1^ny_n+\tilde{Q}_n(x_1,\dots,x_{n-1},y_1,\dots,y_{n-1})$, (homogeneous of bi-degree $n$, if $\deg(x_i):=\deg(y_i):=i$)
\item $\underline{1}=(1)_{n\in\N^*}$ is the neutral element
\end{itemize}
\end{prop}
For $n=2$, we have $Q_2(x_1,x_2,y_1,y_2)=x_2 y_1^2+x_1^2 y_2-x_1^2 y_1^2$, which follows from the above conditions.

The inverse with respect to $\boxtimes$ can be calculated recursively. We have $y_1=\frac{1}{x_1}$, and 
\begin{equation}
\label{inv_Q}
y_n=\frac{1}{x_1^n}(1-\frac{x_n}{x_1^n}-\tilde{Q}_n(x_1,\dots,x_{n-1},y_1(x_{1}),\dots,y_{n-1}(x_1,\dots,x_{n-1}))),
\end{equation}
where the notation $y_{i}(x_1,\dots,x_{i-1})$ expresses that $y_i$ is a function of $x_1,\dots,x_i$, $1\leq i\leq n-1$.

Let us make the following remarks:
\begin{itemize}
\item The homogeneity of the above group laws is due to the causality of the moments, i.e. in order to calculate a moment of order $n$ only information up to order $n$ is needed.
\item
In~\cite{FMcK2015} we generalised  the group laws in Propositions~\ref{PropVRtrafo} and~\ref{PropVStrafo} to the $n$-dimensional setting, thereby extending the derivation in~\cite{FMcK2013}.
\end{itemize}

The next proposition casts~[\cite{V1987} Lemma 2.2] into a more general form.
\begin{prop}
For all $A\in\mathbf{cAlg}_k$ the following holds: 
\begin{enumerate}
\item $(\mathbb{A}^{\N^*}(A),\boxplus)$ is an abelian group with neutral element $\underline{0}=(0,0,0,\dots)$.
\item $(\mathfrak{G}(A),\boxtimes)$ and $(\mathfrak{G}_+(A),\boxtimes)$ are abelian groups with neutral element $\underline{1}=(1,1,1,\dots)$.
\item For the operation $\boxplus$, $\mathbb{A}^{\N^*}$ is a {\em projective limit of finite dimensional groups},  i.e. 
$$
\mathbb{A}^{\N^*}(A)=\varprojlim_{n\in\N^*} \mathbb{A}^n(A).
$$

\item For the operation $\boxtimes$, $\mathfrak{G}(A)$ and $\mathfrak{G}_+(A)$ are {\em projective limits of finite dimensional groups}, i.e. 
$$
\mathfrak{G}(A)=\varprojlim_{n\in\N^*} \mathfrak{G}_n(A)\qquad\text{and}\quad\mathfrak{G}_+(A)=\varprojlim_{n\in\N^*} (\mathfrak{G}_+)_n(A)~.
$$
\end{enumerate}
\end{prop}

\begin{prop}
For all $A\in\mathbf{cAlg}_k$ the following holds:
\begin{enumerate}
\item $(\mathfrak{G}(A),\boxtimes)$ is the direct product of the normal subgroup $\mathfrak{G}_+(A)$ and the $1$-dimensional torus
$\mathfrak{G}_1(A)$, i.e.
$$
\mathfrak{G}(A)=\mathfrak{G}_1(A)\times \mathfrak{G}_+(A)~,
$$
or equivalently, the sequence
\[
\begin{xy}
  \xymatrix{
                             0_A\ar[r]&\mathfrak{G}_1(A)\ar[r]^{\iota} &\mathfrak{G}(A)\ar[r]^p &\mathfrak{G}_+(A)\ar[r] &1_A~.
               }
\end{xy}
\]
with maps $\iota$ and $p$, given by $\iota(a):=(a,1,1,1,\dots)$ and $p(a_1,a_2, a_3\dots):=(1,\frac{a_2}{a_1},\frac{a_3}{a_1},\dots)$, respectively, is exact.
\item The groups $\mathfrak{G}(A)$ and $\mathfrak{G}_+(A)$ are  filtered by their respective subgroups $\mathfrak{G}_n(A)$ and $(\mathfrak{G}_+)_n(A)$ in ascending order.
\end{enumerate}
\end{prop}
In the $n$-dimensional case, $n\geq2$, one has equivalent statements, but instead with the $n$-dimensional torus and the semi-direct product replacing the direct product, cf.~\cite{FMcK2013}. 
\subsection{The natural transformations $R$ and $S$}
\begin{dfthm}[Voiculescu's $R$-transform]
The $R$-transform is a natural isomorphism 
$$
R:(\mathbb{A}^{\N^*},\boxtimes)\rightarrow (\mathbb{A}^{\N^*},+),
$$
i.e. for all $A\in\mathbf{cAlg}_k$ and $f,g\in\mathbb{A}^{\N^*}(A)$, we have 
$$
R_{A}(f\boxplus g)=R_A(f)+R_A(g).
$$
\end{dfthm}
For $A\in\mathbf{cAlg}_k$, the {\bf units} of the ring  of formal power series $A[[z]]$ with coefficient in $A$, are given by
$$
A[[z]]^{\times}=A^{\times}(1+zA[[z]]).
$$
Let $\bullet[[z]]^{\times}$ denote the functor $A\mapsto A[[z]]^{\times}$ and $\Lambda(A)\subset A[[z]]^{\times}$ the subring $\Lambda(A):=1+zA[[z]]$.

To every $f\in\mathfrak{G}(A)$ corresponds a formal power series of the form
$$
f(z)=a_1z+a_2z^2+a_3 z^3+\dots,\qquad a_1\in A^{\times},
$$
which is invertible with respect to {\bf composition of power series}. We denote this inverse by $f^{-1}(z)$, and in fact, we have a natural transformation $\operatorname{inv}:\mathfrak{G}\rightarrow\mathfrak{G}$, which is an involution, i.e. $\operatorname{inv}^2=\id$ and hence  $f^{-1}=\operatorname{inv}(f)$.

\begin{dfthm}[Voiculescu's $S$-transform]
The {\bf $S$-transform} is the natural isomorphism (of set valued functors)
$$
S:\mathfrak{G}\rightarrow\bullet[[z]]^{\times},
$$ 
given by 
\begin{equation}
\label{S_trafo}
S_A(f):=\frac{1+z}{z}f^{-1}(z),
\end{equation}
for $A\in\mathbf{cAlg}_k$ and $f\in\mathfrak{G}(A)$.
\end{dfthm}

The following statements, originally deduced by Voiculescu in~\cite{V1987}, cf. also~\cite{VDN}, not only justify the above definitions but further extend them. 

\begin{prop}
The $S$-transform is a natural isomorphism between the group-valued functors $\mathfrak{G}$ and $\bullet[[z]]^{\times}$,  
i.e. for $A\in\mathbf{cAlg}_k$, and the abelian groups $((A[[z]])^{\times},\cdot)$, $(\Lambda(A),\cdot)$, $(\mathfrak{G}(A),\boxtimes)$  and $(\mathfrak{G}_+(A),\boxtimes)$, the diagram 
\[
\begin{xy}
  \xymatrix{
 (\mathfrak{G}_+(A),\boxtimes)\ar[r]_{S_A}^{\sim}  \ar[d]_{\operatorname{incl.} }  & (\Lambda(A),\cdot)\ar[d]_{\operatorname{incl.}}\\
     (\mathfrak{G}(A),\boxtimes)\ar[r]_{S_A}^{\sim}       & ((A[[z]])^{\times},\cdot)
               }
\end{xy}
\]
commutes. For $f,g\in\mathfrak{G}(A)$, it implies:  
$
S_A(f\boxtimes g)=\mathcal{S}_A(f)\cdot S_A(g).
$
\end{prop}

We close this subsection by putting the $R-$ and $S$-transform into perspective by comparing them to the integral transforms in classical probability theory. 
The {\bf Mellin transform}, cf. e.g.~\cite{GS}, of a positive real random variable $X$, i.e. $X\geq0$, is defined as
$$
\mathcal{M}_X(t):=\E[X^{it}]=\int_0^{\infty} x^{it}d\mu_X(x),
$$
where $t\in\R$ and with the convention $0^{it}:=0$, for all $t$.
Then, for $X,Y$ independent classical and $a,b$ free non-commutative random variables, respectively, the following relations hold:
\[
\begin{tabular}{l|l|l} $X,Y$~{\bf independent}& convolution & transform (Fourier / Mellin)  \\\hline additive & $\mu_{X+Y}=\mu_X\ast\mu_Y$ & $\mathcal{F}_{X+Y}(t)=\mathcal{F}_X(t)\cdot\mathcal{F}_Y(t)$  \\multiplicative & $\mu_{X\cdot Y}=\mu_X\ast_{\operatorname{m}}\mu_Y$ & $\mathcal{M}_{X\cdot Y}(t)=\mathcal{M}_X(t)\cdot\mathcal{M}_Y(t)$   \\\hline\hline 
$a,b$~{\bf free} & convolution & transform ($R$ / $S$)  \\\hline additive & $\mu_{a+b}=\mu_a\boxplus\mu_b$ & $R_{a+b}(z)=R_a(z)+R_b(z)$  \\multiplicative & $\mu_{a\cdot b}=\mu_a\boxtimes\mu_b$ & $S_{a\cdot b}(z)=S_a(z)\cdot S_b(z)$ \end{tabular}
\]
where we used the notation $R_a(z):=R_{\mu_a}(z)$ and $S_a(z):=S_{\mu_a}(z)$, etc.

\subsection{Free cumulants and the $\boxtimes_{\operatorname{NS}}$-convolution}
The theory of free cumulants was originally created by R.~Speicher~\cite{S1997}, and then further  developed  jointly with A.~Nica, culminating in the monograph~\cite{NS}. An essential role is played by non-crossing partitions and their Kreweras complement, cf.~[\cite{NS}, Lectures 9., 17. and 18.].

For $n\in\N^*$, let $\operatorname{NC}(n)$ denote the poset of non-crossing partitions of the set $\{1,\dots,n\}$.
For $A\in\mathbf{cAlg_k}$, $f,g\in\mathfrak{G}(A)$, i.e. $f=(a_1,a_2,a_3,\dots)$ and $g=(b_1,b_2,b_3,\dots)$ with $a_1,b_1\in A^{\times}$,  and a non-crossing partition $\pi=\{V_1,\dots, V_p\}\in\operatorname{NC}(n)$, with {\bf Kreweras complement} $K(\pi)=\{W_1,\dots, W_q\}$, one introduces the operators  $X_{w,\pi}$ and $X_{w,K(\pi)}$, which are defined as follows:
\begin{eqnarray*}
\label{}
X_{\pi}(f)&:=&a_{|V_1|}\cdots a_{|V_p|}~,\\
X_{K(\pi)}(g)&:=& b_{|W_1|}\cdots b_{|W_q|}~,\\
\end{eqnarray*}
where $|V_i|$ and $|W_j|$ denote the cardinality of the set $V_i$ and $W_j$, respectively.
\begin{df}[Nica $\&$ Speicher's free boxed convolution]
For $A\in\mathbf{cAlg_k}$ and $f,g\in\mathfrak{G}(A)$, the one-dimensional {\bf free boxed convolution} $\boxtimes_{\operatorname{NS}}$ is defined as
\begin{equation}
\label{}
X_{n}(f\boxtimes g):=\sum_{\pi\in\operatorname{NC}(n)}X_{\pi}(f)\otimes X_{K(\pi)}(g)=\sum_{\begin{subarray}{c}
       \pi\in\operatorname{NC}(n)\\\pi=\{V_1,\dots V_p\}\\K(\pi)=\{W_1,\dots W_q\}
       \end{subarray}} a_{|V_1|}\cdots a_{|V_p|} b_{|W_1|}\cdots b_{|W_q|},
\end{equation}
with the product on the right hand-side taken in $A$.
\end{df}
The special vectors 
$\operatorname{Zeta}:=(1,1,1,\dots)$ and $\operatorname{Moeb}:=(\operatorname{Catalan}_n)_{n\in\N^*}$, where $\operatorname{Catalan}_n$ is the $n$th Catalan number, define the following maps:
\begin{eqnarray*}
\varphi_m(\bullet):=\bullet\boxtimes_{\operatorname{NS}} \operatorname{Zeta}\\
\varphi_{fc}(\bullet):=\bullet\boxtimes_{\operatorname{NS}}\operatorname{Moeb}
\end{eqnarray*}
which, by~\cite{NS}, satisfy
$$
\varphi_m\circ\varphi_{fc}=\id,\quad  \varphi_{fc}\circ\varphi_{m}=\id. 
$$
Therefore the two maps $\varphi:A^{\N^*}\rightarrow A^{\N^*}$ form a pair of isomorphisms, and give co-ordinate changes from moments to free cumulants. 
\begin{prop}
The one-dimensional free boxed convolution $\boxtimes_{\operatorname{NS}}$, defines an infinite commutative algebraic group law over $\Z$, i.e. there exists a countable sequence of universal polynomials $K_{\boxtimes_{\operatorname{NS}}}=K=(K_1,K_2,K_3,\dots)$, $K_i\in\Z[x_1,\dots,x_n,y_1,\dots,y_n]$ such that
\begin{itemize}
\item $K_1(x_1,y_1)=x_1\cdot y_1$,
\item $K_n(x_1,\dots, x_n,y_1,\dots,y_n)=K_n(y_1,\dots,y_n,x_1,\dots, x_n)$, (symmetric, commutative)
\item $K_n(x_1,\dots,x_n,y_1,\dots,y_n)=x_ny_1^n+x_1^ny_n+\tilde{K}_n(x_1,\dots,x_{n-1},y_1,\dots,y_{n+1})$ with
$$
\tilde{K}_n=\sum_{\begin{subarray}{c}
       \pi\in\operatorname{NC}(n),\pi\neq\mathbf{1,0}\\\pi=\{V_1,\dots V_p\}\\K(\pi)=\{W_1,\dots W_q\}
       \end{subarray}} x_{|V_1|}\cdots x_{|V_p|} y_{|W_1|}\cdots y_{|W_q|},
$$
where $\mathbf{0}$ and $\mathbf{1}$ denote the minimal and maximal element of the poset $\operatorname{NC}(n)$. It is homogeneous of bi-degree if $\deg(x_i):=\deg(y_i):=i$, $i\in\N^*$.
\item $(1,\underline{0})=(1,0,0,0,\dots)$ is the neutral element.
\end{itemize}
\end{prop}
For $n=2$, with two non-crossing partitions, we have
$K_2(x_1,x_2,y_1,y_2)=x_2y_1^2+x_1^2y_2$
and for $n=3$, where one has five non-crossing partitions, we have
\begin{eqnarray*}
K_3(x_1,x_2,x_3,y_1,y_2,y_3)&=&x_3y_1^3+x_1^3y_3+x_1x_2y_1y_2+x_1x_2y_1y_2+x_1x_2y_1y_2\\
&=&x_3y_1^3+x_1^3y_3+3x_1x_2y_1y_2.
\end{eqnarray*}
The inverse with respect to $\boxtimes_{\operatorname{NS}}$ can be calculated recursively. We have $y_1=\frac{1}{x_1}$ and
$$
y_n=-\frac{x_n}{x_1^{2n}}-\frac{1}{x_1^n}\tilde{K}_n(x_1,\dots,x_{n-1},y_1(x_1),\dots,y_{n-1}(x_1,\dots,x_{n-1})),
$$
where $y_i(x_1,\dots,x_{i-1})$ expresses that $y_i$ is a function of $x_1,\dots,x_i$, $1\leq i\leq n-1$.
\begin{proof}
From [\cite{NS} Proposition 17.5, 1)] it follows that the group law is well defined, in particular associative. The existence and form of the unit from can be deduced from~[\cite{NS} Proposition 17.5, 2)].
In oder to show the claim about the bi-degree, one has to use the the following property of the Kreweras  complement, i.e. for $\pi\in\operatorname{NC}(n)$ one has 
$$
|\pi|+|K(\pi)|=n+1,
$$
and then one proceeds as in the equivalent proof in~\cite{FMcK2013}. 
\end{proof}
\begin{dfthm}[Nica and Speicher's free $F$-transform]
The free $F$-transform is the natural isomorphism (of set valued functors)
$$
F:\mathfrak{G}\rightarrow \bullet
[[z]]^{\times},
$$
given by
\begin{equation}
\label{F-trafo}
F_A(f):=\frac{1}{z}f^{-1}(z),
\end{equation}
for all $A\in\mathbf{cAlg_k}$ and $f\in\mathfrak{G}(A)$.
\end{dfthm}
The following two statements can be deduced from~[\cite{NS}, Lecture 18.]
\begin{prop}
The free $F$-transform is a natural isomorphism from $\mathfrak{G}$ to $\bullet[[z]]^{\times}$, i.e. every ${F}_A$ is an isomorphism of abelian groups, such that the diagram
 \[
\begin{xy}
  \xymatrix{
  \left(\mathfrak{G}_+(A),\boxtimes_{\operatorname{NS}}\right)\ar[d]_{\operatorname{incl.}}\ar[rr]_{{F_A}}^{\sim}&&\left(\Lambda(A),\cdot\right)\ar[d]_{\operatorname{incl.}}\\
\left(\mathfrak{G}(A),\boxtimes_{\operatorname{NS}}\right)\ar[rr]_{F_A}^{\sim}&&\left(A[[z]]^{\times},\cdot\right) 
     }
\end{xy}
\]
commutes. For $f,g\in\mathfrak{G}(A)$, it implies:
$\mathcal{F}_R(f\boxtimes_{\operatorname{NS}} g)=\mathcal{F}_A(f)\cdot \mathcal{F}_A(g).$
\end{prop}
\subsection{Hopf algebras related to free probability}
Let $G$ be a covariant and representable functor from $\mathbf{cAlg}_k$ to the category of groups. Then $G$ defines an affine group scheme. This is equivalent to the existence, up to isomorphism, of a commutative but not necessarily co-commutative Hopf algebra $H$, such that $\Hom_{\mathbf{cAlg}_k}(H,A)$ is naturally isomorphic to $G(A)$, for all $A\in\mathbf{cAlg}_k$. As usual, $\mathbb{G}^n_a$ denotes the $n$-dimensional additive group law and $\mathbb{G}^n_m$ the n-dimensional multiplicative group law.
\begin{prop}
The formal group law $F_{\boxtimes}$ corresponding to $Q_{\boxtimes}$, is given by
$$
F_{\boxtimes,n}(x_1,\dots,x_n, y_1,\dots, y_n)=Q_n(x_1+1,\dots,x_n+1, y_1+1,\dots, y_n+1)-1,
$$
for $n\in\N^*$. 

The formal group law $F^+_{\boxtimes}$, corresponding to $x_1=y_1=1$ in $Q_{\boxtimes}$,  is homogeneous of degree $n-1$, if $\operatorname{deg}(x_i):=\operatorname{deg}(y_i):=i-1$, for $i\geq2$.
\end{prop}
\begin{proof}
The formal group law is obtained by a change of co-ordinates, which in our case is translation by $-\underline{1}$, i.e. $\underline{x}\mapsto \underline{x}-\underline{1}$ for $\underline{x}\in\mathfrak{G}$. 

Then the claim follows from $F_{\boxtimes}(\underline{x},\underline{y})=((\underline{x}+\underline{1})\boxtimes(\underline{y}+\underline{1}))-\underline{1}$, by projecting onto the individual components. 
\end{proof}
Let us consider the following examples for $F_{\boxtimes}$ and $F_{\boxtimes}^+$, where we set $x_1=y_1=0$ in the latter case:

For $n=1$, we have
$F_{\boxtimes,1}(x_1,y_1)=x_1+y_1+x_1y_1$, which is the multiplicative group law, and $F^+_{\boxtimes,1}\equiv0$.

For $n=2$, we have
$$
F_{\boxtimes,2}(x_1,x_2,y_1,y_2)=x_2+y_2+x_2 y_1^2+2x_2y_1+x_1^2y_2+2x_1y_2-x_1^2y_1^2-2x_1^2y_1-2x_1y_1^2-4x_1y_1.
$$ 
and $F^+_{\boxtimes,2}=x_2+y_2$.

For $n=3$, we have
\begin{eqnarray*}
F_{\boxtimes,3}(x_1,x_2,x_3,y_1,y_2,y_3)&=&x_3y_1^3+x_1^3y_3+x_2x_1y_2y_1-x_2x_1y_1^3+x_2x_1y_1y_2-x_1^3y_1y_2\\
&&-x_1x_2y_1^3-x_1^3y_2y_1-x_2x_1y_1^3+x_1x_2y_2y_1-x_1^3y_2y_1+2x_1^3y_1^3\\
&=&x_3y_1^3+x_1^3y_3+3x_1x_2y_1y_2-3x_1x_2y_1^3-3x_1^3y_1y_2+2x_1^3y_1^3.
\end{eqnarray*}
and  $F_{\boxplus,3}^+=x_3+y_3+3x_2y_2$.

The free polynomial algebra $k[x_1,x_2,x_3,\dots]$, in countably many commuting variables $x_i$ is an $\N$-graded algebra, with $\operatorname{deg}(x_i):=i$, $i\in\N^*$, i.e. $k[x_1,x_2,x_3,\dots]=\bigoplus_{n=0}^{\infty} H_n$, where $H_0= k\cdot 1$ and for $n\geq 1$, $H_n$  is the $k$-linear span by all monomials $x_{i_1}\cdots x_{i_m}$ of degree $i_1+\cdots+ i_m=n$. As an algebra it is graded connected, since $H_i\cdot H_j\subsetneq H_{i+j}$. Further, we denote the ring of Laurent polynomials by $k[x^{-1},x]$
and we shall use the convention $x_i:=x_1\otimes 1$ and $y_i:=1\otimes x_1$ in the following formul\ae.

\begin{thm}
The free additive and multiplicative convolution define pro-affine group schemes.
\begin{enumerate}
\item
$(\mathbb{A}^{\N^*},\boxplus)$ is a pro-unipotent affine group scheme. It is represented by 
the graded connected Hopf algebra $H_{\boxplus}=k[x_1,x_2,x_3,\dots]$, with the co-product given by
$$
\Delta_{\boxplus}(x_i):=P_n(x_1,\dots,x_n,y_1,\dots, y_n),
$$
and the co-unit $\varepsilon(x_i)=0$, $i\in\N^*$. The antipode $\alpha$ is calculated recursively, starting with $\alpha(x_1)=-x_1$.
\item
$(\mathfrak{G},\boxtimes)$ is a pro-affine group scheme, which is represented by 
the filtered Hopf algebra $H_{\boxtimes}=k[x_1^{-1},x_1,x_2,x_3,\dots]$, with the co-product given by
$$
\Delta_{\boxtimes}(x_i)=Q_n(x_1,\dots,x_n,y_1\dots,y_n),
$$
where $x_1$ and $x_1^{-1}$ are group-like elements,
i.e. $\Delta_{\boxtimes}(x_1^{(-1)})=x_1^{(-1)}\otimes x_1^{(-1)}$. The co-unit satisfies $\varepsilon(x_1)=1_k$ and $\varepsilon(x_i)=0$, for $i\geq2$. The antipode is calculated recursively, starting with $\alpha(x_1)=x_1^{-1}$.
\item
$(\mathfrak{G}_+,\boxtimes)$ is a pro-unipotent affine group scheme. It is representable by 
the graded connected Hopf algebra $H^+_{\boxtimes}=k[x_2,x_2,x_3,\dots]$, 
$\deg(x_i):=i-1$, for $i\geq2$, and with the co-product given by
$$
\Delta^+_{\boxtimes}(x_i)=F^+_n(x_2,\dots,x_n,y_2,\dots, y_n).
$$
The co-unit satisfies $\varepsilon(x_i)=0$, $i\geq2$, and the antipode is calculated recursively, starting with $\alpha(x_2)=-x_2$.
\end{enumerate}
\end{thm}

\begin{thm}[Algebraic linearisation of $\boxtimes$] 
\label{main-iso}
There exists a natural isomorphism
$$
\log_{\boxtimes}:(\mathfrak{G}_+,\boxtimes)\rightarrow\mathbb{G}_a^{\N^*},
$$ 
which is given by
\begin{equation}
\label{Log_S_trafo}
\log_{\boxtimes}=\left(z\frac{d}{dz}\ln\right)\circ S,
\end{equation}
where $S$ is the $S$-transform~(\ref{S_trafo}). Hence, for all $A\in\mathbf{cAlg_k}$ and $f,g\in\mathfrak{G}_+(A)$ the following holds:
$$
{\log_{\boxtimes}}_A(f\boxtimes g)={\log_{\boxtimes}}_A(f)+{\log_{\boxtimes}}_A(f).
$$
\end{thm}
\begin{cor}[Algebraic linearisation of $\boxtimes_{\operatorname{NS}}$]
There exists a natural isomorphisms
$$
\log_{\boxtimes_{\operatorname{NS}}}:(\mathfrak{G}_+,\boxtimes_{\operatorname{NS}})\rightarrow\mathbb{G}_a^{\N^*},
$$ 
which is given by
\begin{equation}
\label{Log_F_trafo}
\log_{\boxtimes_{\operatorname{NS}}}=\left(z\frac{d}{dz}\ln\right)\circ F,
\end{equation}
where $F$ is the free $F$-transform~(\ref{F-trafo}). Hence, for all $A\in\mathbf{cAlg_k}$ and $f,g\in\mathfrak{G}_+(A)$ the following holds:
$$
{\log_{\boxtimes_{\operatorname{NS}}}}_A(f\boxtimes_{\operatorname{NS}} g)={\log_{\boxtimes_{\operatorname{NS}}}}_A(f)+{\log_{\boxtimes_{\operatorname{NS}}}}_A(g).
$$
\end{cor}
\begin{rem}
The version  Mastnak and Nica have given in~\cite{MN} is for $\boxtimes_{NS}$. 

In higher dimensions the free multiplicative convolution of cumulants, as originally introduced by Nica and Speicher~\cite{NS} is not abelian and therefore not linearisable, cf.~\cite{FMcK2013,FMcK2015}.
\end{rem}

The situation is different for the free additive convolution, which is commutative in every dimension as it is given by a commutative formal group law.  Therefore, there exists always a natural  isomorphism  
$\log_{\boxplus}:(\mathbb{A}^{\N^*},\boxplus)\rightarrow\mathbb{G}_a^{\N^*}$ which is, up to a linear transformation, given by a lower triangular matrix. In the one-dimensional case it corresponds to Voiculescu's $R$-transform~\cite{V1985,V1986}, i.e.
$$
\log_{\boxplus}\sim R.
$$ 
For an alternative Lie theoretic derivation in the general case, one should consider~\cite{FMcK2015}.
\begin{cor}
There exists a natural isomorphism of pro-unipotent affine group schemes
\begin{equation}
\operatorname{EXP}:(\mathbb{A}^{\N^*},\boxplus)\rightarrow(\mathfrak{G}_+,\boxtimes), 
\end{equation}
which is given by
\begin{equation}
\operatorname{EXP}:=S^{-1}\circ(z\frac{d}{dz}\ln)^{-1}\circ R,
\end{equation}
where $R$ and $S$ are the $R-$ and $S$-transform, respectively. For $A\in\mathbf{cAlg}$ and $f,g\in\mathbb{A}^{\N^*}(A)$ the following holds:
$$
\operatorname{EXP}_A(f\boxplus g)=\operatorname{EXP}_A(f)\boxtimes\operatorname{EXP}_A(g).
$$
\end{cor}   

\section{Co-rings and comonads}
\subsection{Witt vectors}
In this subsection, we use the following references~\cite{I1979,H,MuB,Y}.  

\begin{prop}[Witt ring]
There exist universal polynomials $S_{\W}=({S_{\W}}_n)_{n\in\N^*}$ and $P_{\W}=({P_{\W}}_n)_{n\in\N^*}$ with ${S_{\W}}_n,{P_{\W}}_n\in\Z[x_1,\dots,x_n,y_1,\dots,y_n]$ 
which satisfy:
\begin{itemize}
\item $S_1(x_1,y_2)=x_1+y_1$
\item $S_n(x_1,\dots,x_n,y_1,\dots,y_n)=S_n(y_1,\dots,y_n,x_1,\dots,x_n)$ (commutative / symmetric)
\item $S_n(x,y)=x_n+y_n+\tilde{S}_{\W}(x_1,\dots,x_{n-1},y_1,\dots,y_{n-1})$ (formal group law)
\item $P_1(x_1,y_1)=x_1\cdot y_1$
\item $P_n(x_1,\dots,x_n,y_1,\dots,y_n)=P_n(y_1,\dots,y_n,x_1,\dots,x_n)$ (commutative)
\end{itemize}
\end{prop}
Let $\W$ denote the functor $\W:\mathbf{cAlg_k}\rightarrow\mathbf{cRing_k}$, whose composition with the forgetful functor to $\mathbf{Set}$, is $\W(A):=A^{\N^*}$.
\begin{prop}
$(\W(A),+_{\W},\cdot_{\W},(0,0,0,\dots),(1,0,0,0,\dots))$ is a commutative unital ring with multiplicative unit $(1,0,\dots,0,\dots)$, and with 
the {\bf addition} and {\bf multiplication} defined by:
\begin{eqnarray}
\label{Ws}
({x}+_{\W}{y})_n & := & {S_{\W}}_n({x},{y}), \\
\label{Wm}
({x}\cdot_{\W}{y})_n & := & {P_{\W}}_n({x},{y}),
\end{eqnarray}
\end{prop}
The polynomials $S_{\W}$ and $P_{\W}$ can be determined by the requirement that the {\bf ghost map} 
\begin{equation}
w_A:\mathbb{W}(A)\rightarrow A^{\N^*}
\end{equation}
is a natural isomorphism of commutative rings. The map $w$ is given by a vector of maps  $w=(w_n)_{n\in\N^*}$, with components 
\begin{eqnarray}
\label{ghost_map}
w_n&:&\W(A)\rightarrow A,\nonumber\\
{x}&\mapsto& w_n(x_1,\dots,x_n):=\sum_{d|n}d x_d^{n/d},
\end{eqnarray} 
which are ring homomorphisms. The $w_n(x_1,\dots,x_n)\in\Z[x_1,\dots,x_n]$ are the {\bf Witt polynomials}.
From formula~(\ref{ghost_map}), one obtains for $n=1,\dots,5$:
\begin{eqnarray*}
w_1&=&x_1,\\
w_2&=&x_1^2+2x_2,\\
w_3&=&x_1^3+3x_3,\\
w_4&=&x_1^4+2x_2^2+4x_4,\\
w_5&=&x_1^5+3x_5.
\end{eqnarray*}
\begin{prop}[\cite{MuB}]
The Witt ring is represented by $\Z[w_1,\dots,w_n,\dots]$, with the ``co-addition" and co-multiplication given by
\begin{eqnarray*}
\Delta_+(w_n)&:=&w_n\otimes1+1\otimes w_n,\\
\Delta_{\cdot}(w_n)&:=&w_n\otimes w_n.
\end{eqnarray*} 
\end{prop}
The above Proposition states that the ring structure is not defined by a Hopf algebra,  but by a co-ring, which translates into component-wise addition and multiplication.
\begin{prop}
For $A\in\mathbf{cAlg_k}$, the
ghost map defines a natural isomorphism of ring valued functors, i.e.
$$
w_A:(\W(A),+_{\W},\cdot_{\W})\rightarrow(A^{\N^*},+,\cdot_{\operatorname{H}})
$$ 
is a natural isomorphism of commutative unital rings.
\end{prop}
\subsection{The $\lambda$-ring of formal power series}
References for this part are~\cite{FMcK2012,H,Y}.
The functor $\Lambda:\mathbf{cRing}_k\rightarrow\mathbf{cRing}_k$, when composed with the forgetful functor to $\mathbf{Ab}$, the category of abelian groups, has 
$$
\Lambda(A):=(1+zA[[z]])^{\times}=\{f(z)=1+a_1z+a_2z^2+\dots~|~a_i\in A\}
$$
as the underlying abelian group, with the {``addition"} $+_{\Lambda}$ given by the usual multiplication of power series, i.e. $+_{\Lambda}:=\cdot$ and the ``zero" equal to $1$.

In order to have a ring structure, the {``multiplication"} $\cdot_{\Lambda}$, is defined as follows: for $f,g\in\Lambda(A)$ consider the formal factorisations 
$f(z)=\prod_{i=1}^{\infty}(1-x_iz)^{-1}$ and $g(z)=\prod_{j=1}^{\infty}(1-y_jz)^{-1}$,
and define 
$$
(1-at)^{-1}\cdot_{\Lambda}(1-bt)^{-1}:=(1-abt)^{-1},
$$ 
which is then extended bilinearily, i.e. one uses
$$
\bullet_1\cdot_{\Lambda}(\bullet_2+_{\Lambda}\bullet_3)=\bullet_1\cdot_{\Lambda}\bullet_2+_{\Lambda}\bullet_1\cdot_{\Lambda}\bullet_3.
$$
For every $A\in\mathbf{cAlg}_k$, the {\bf ghost map}
\begin{equation}
z\frac{d}{dz}\ln:\Lambda(A)\rightarrow A^{\N^*},
\end{equation}
is an isomorphism of commutative rings, cf. e.g.~\cite{I1979,H,MuB}. It is well known that $\Lambda$, as a group-valued functor, is representable by the Hopf algebra of complete symmetric functions $\mathbf{Symm}$, cf~\cite{H}. 
\begin{prop}[e.g.\cite{I1979,H,Y}]
There exists a natural isomorphism, the {\bf Artin-Hasse exponential map}, $E:\W\rightarrow\Lambda$, which for $A\in\mathbf{cAlg_k}$, is given by
$$
E_A:\W(A)\rightarrow(A),\quad (x_n)_{n\in\N^*}\mapsto\prod_{n=1}^{\infty}\frac{1}{1-x_nz^n}.
$$
\end{prop}
\begin{prop}[Relation with multiplicative free convolution]
There exists a natural isomorphism 
of pro-affine group schemes $\varphi:(\W,+_{\W})\rightarrow(\mathfrak{G}_+,\boxtimes)$,  which for $A\in\mathbf{cAlg}_k$, is given by
$$
\varphi_A=\operatorname{inv}\left(\frac{z}{1+z}\prod_{n=1}^{\infty}\frac{1}{1-a_nz^n}\right),
$$
where $\operatorname{inv}$ denotes the compositional inverse of power series. Hence, we have
$$
\varphi_A(a+_{\W}b)=\varphi_A(a)\boxtimes\varphi_A(b).
$$
\end{prop}
\begin{proof}
The isomorphism  $E:\W\rightarrow\Lambda$ is given by, cf. e.g.~\cite{H,Y}
$$
(a_n)_{n\in}\mapsto\prod_{n=1}^{\infty}(1-a_nz^n)^{-1}
$$
and then apply the inverse $S^{-1}$ of the $S$-transform, in order to obtain the result.
\end{proof}
\subsection{Differential algebras}
The references for this subsection are~\cite{GaSt,GK,ZGK}. Further, we shall assume all algebras to be commutative and unital.

\begin{df}[\cite{GK}]
Let $A\in\mathbf{cAlg}_k$, $\lambda\in k$ and $d\in\Hom_k(A,A)$, i.e. $d$ is a $k$-linear endomorphism.
A {\bf differential $k$-algebra of weight $\lambda$} is a pair $(A,d)$ such that for all $x,y\in A$ the following holds:
\begin{equation}
\label{DA}
d(xy)=d(x)y+x d(y)+\lambda d(x)d(y)\quad\text{and}\quad d(1)=0.
\end{equation}
The operator $d$ is called a {\bf derivation of weight $\lambda$}.
\end{df}
\begin{df}[\cite{GK}]
For $A\in\mathbf{cAlg}_k$, the {\bf $\lambda$-weighted Hurwitz product} on $A^{\N}$ is given by
\begin{equation}
(f\cdot_{\lambda}g)_n:=\sum_{k=0}^n\sum_{j=0}^{n-k}
\binom{n}{k}\binom{n-k}{j}
\lambda^k a_{n-j}b_{k+j}.
\end{equation}
\end{df}

\begin{prop}[\cite{GK}, Section 2.3.]
For every $A\in\mathbf{cAlg}_k$, $(A^{\N},+,\cdot_{\lambda},(1_A,0,0,0,\dots))$ is an associative and commutative $k$-algebra with $(1_A,0,0,0,\dots)$ as $\cdot_{\lambda}$-multiplicative unit. In particular, for $a,b,c\in A^{\N}$ 
\begin{equation}
(a+b)\cdot_{\lambda} c=(a\cdot_{\lambda}c)+(b\cdot_{\lambda}c),
\end{equation}
i.e. distributivity holds.
\end{prop}
We have a natural transformation $d_{A}:A^{\N}\rightarrow A^{\N}$ which is defined as follows: for $a\in A^{\N}$ set
\begin{equation}
(d_A(a))_n:=a_{n+1},
\end{equation}
i.e. $d_A$ acts by shifting a sequence by one position to the left (décalage). 
\[
\begin{tabular}{l|r|l} operator& action on a sequence in $A^{\N}$& name   \\\hline 
$\id_{A^{\N}}$ & $a_0,a_1,a_2,a_3,a_4,\dots$ & identity \\$d_A$ & $a_1,a_2,a_3,a_4,a_5,\dots$ & $\lambda$-derivation (décalage)
\end{tabular}
\]

\begin{prop}
\label{H-linear}
$(\mathfrak{G}_+(A),\cdot_{\lambda},(1,0,0,0,\dots))$ is a commutative pro-unipotent affine group scheme, which for $\lambda\in\Q^*$ is isomorphic to $(A^{\N},+)$.
\end{prop}
The above Proposition~\ref{H-linear} shows the existence of a linearising isomorphism, which can be calculated recursively. 

In fact we have, summarising [\cite{GK}, Propositions~$2.7$ and $2.10$] and~\cite{GaSt},

\begin{prop}[\cite{GaSt,GK}]
For $A\in\mathbf{cAlg}_k$, we have:
\begin{enumerate}
\item 
The shift map $d_A$ is a $\lambda$-derivation for the Hurwitz product $\cdot_{\lambda}$.
\item There exists a morphism of commutative rings
$$
\gamma:(A^{\N},\cdot_{\lambda})\rightarrow (A^{\N},\cdot_{H}),
$$
which is given by
$$
\gamma(a)_n:=\sum_{j=0}^n\binom{n}{j}\lambda^ja_j.
$$
If $\lambda\in k^{\times}$, i.e. $\lambda$ is {\bf invertible} then $\gamma$ is an {\bf isomorphism} of commutative rings. 
\end{enumerate}
\end{prop}
For $x,y\in A$, let 
\begin{equation}
\Gamma_{\lambda}(x,y):=d(xy)-d(x)y-xd(y)
\end{equation}
\begin{prop}[Carré du champ]
For every differential $k$-algebra of weight $\lambda$, $(A,d)$, $\Gamma_{\lambda}:A\otimes_k A\rightarrow A$, is symmetric and bilinear. The quantity $\Gamma_{\lambda}(x,y)=\lambda d(x)d(y)$ measures the deviation of $d$ being a derivation on $xy$.
\end{prop}
\subsection{Endomorphism algebras}
Let us consider the category of algebras with a fixed endomorphism, cf.~\cite{Bel,GaSt,Joy1}. More precisely, we consider pairs $(A,\sigma)$ as objects with $A\in\mathbf{cAlg}_k$ and $\sigma\in\Hom_{\mathbf{cAlg}_k}(A,A)$ and we denote the resulting category by $\mathbf{EndcAlg}_k$.
The initial object is $(k,\id_k)$ and the forgetful functor $U$ preserves the initial object, as $U(k,\id_k)=k$.

The free algebra generated by one element $t_1$ in $\mathbf{EndcAlg}_k$, is
\begin{eqnarray*}
k[t_1]_{\sigma}&=&k[t_1,t_2,t_3,\dots],\\
\sigma(t_n)&=&t_{n+1}, n\in\N^*,
\end{eqnarray*}
where $t_n$, $n\in\N^*$ are commuting variables. The right adjoint $F_{\sigma}$ to $U$, is given by
$$
F_{\sigma}(A)=(A^{\N^*},\text{$+,\cdot$ point-wise},\partial),
$$ 
where 
$$
\partial(a_1,a_2,a_3,a_4,\dots):=(a_2,a_3,a_4,\dots),
$$ 
is the {\bf décalage operator}~\cite{Bel,GaSt} or {\bf Frobenius endomorphism}, cf.~\cite{I1979}.

The corresponding adjunction $(U,F_{\sigma},\eta^{\sigma},\varepsilon^{\sigma})$ is given by 
\begin{eqnarray*}
\eta^{\sigma}(A,\sigma)&\rightarrow& F_{\sigma}U(A,\sigma)=(A^{\N^*},\text{$+,\cdot$ point-wise},\partial)\\
\eta^{\sigma}(a)&=&(a,\sigma(a),\sigma^2(a),\sigma^3(a),\dots),\quad a\in A,\\
\varepsilon^{\sigma}_A:UF_{\sigma}(A)=A^{\N^*}&\rightarrow& A,\\
\varepsilon^{\sigma}_A(\underline{a})&:=&a_1,\quad \underline{a}\in A^{\N^*}.
\end{eqnarray*}
The {\bf comonad} $C:\mathbf{cAlg}_k\rightarrow\mathbf{cAlg}_k$ corresponding to the above adjunction $C:=(T,\varepsilon,\delta)$ is given by
\begin{eqnarray}
\label{delta_sigma}
\nonumber
T&:=&UF_{\sigma}\\
\delta&:=&U\eta^{\sigma} F_{\sigma}\\\nonumber
\delta_A&:&A^{\N^*}\rightarrow (A^{\N^*})^{\N^*}\cong A^{\N^*\times\N^*}\\\nonumber
(\delta_A(\underline{a}))_n&=&(\underline{a},\partial(\underline{a}),\partial^2(\underline{a}),\partial^3(\underline{a}),\dots),\quad\underline{a}\in A^{\N^*},\nonumber
\end{eqnarray}
i.e. $(\delta_A(\underline{a}))_{m,n}=\underline{a}_{m+n-1}$, $m,n\in\N^*$. 
\begin{prop}
The co-comparison functor $K:\mathbf{EndcAlg}_k\rightarrow{\mathbf{cAlg}_k}_C$, where ${\mathbf{cAlg}_k}_C$ is the category of $C$-coalgebras, is an isomorphism. Therefore, $\mathbf{EndcAlg}_k$ is comonadic over $\mathbf{cAlg}_k$.
\end{prop}
\begin{proof}
This corresponds to~[\cite{GaSt} Proposition 4.2].
\end{proof}
\subsection{Free convolution coalgebras}
For $a,b\in A^{\N^*}$, let 
\begin{equation*}
a\boxdot b:=R^{-1}(R(a)\cdot_{\operatorname{H}} R(b)),
\end{equation*}
where $\cdot_{\operatorname{H}}$ is the point-wise multiplication,
and for $a,b\in \mathfrak{G}_+(A)$, let 
\begin{equation*}
a\square\hspace{-0.85 em}\ast b:=S^{-1}(S(a)\cdot_{\Lambda} S(b)).
\end{equation*}
\begin{prop} 
For every $A\in \mathbf{cAlg}_k$, 
$(A^{\N^*},\boxplus,\boxdot)$ and 
$(\mathfrak{G}_+(A),\boxtimes, {\square\hspace{-0.6 em}\ast})$ are commutative unital rings  with neutral element $1$ and $S^{-1}(1-t)$, respectively. 

There exists a natural isomorphism $\operatorname{LOG}:(\mathfrak{G}_+(A),\boxtimes, {\square\hspace{-0.6 em}\ast})\rightarrow(A^{\N^*},\boxplus,\boxdot)$, which is  given by
$$
\operatorname{LOG}:=R^{-1}\circ(z\frac{d}{dz}\ln)\circ S.
$$
\end{prop}
\begin{thm} Let $\lambda\in k^{\times}$ and $A\in\mathbf{cAlg}_k$. The following rings schemes, as depicted in the diagram below, are all isomorphic :  

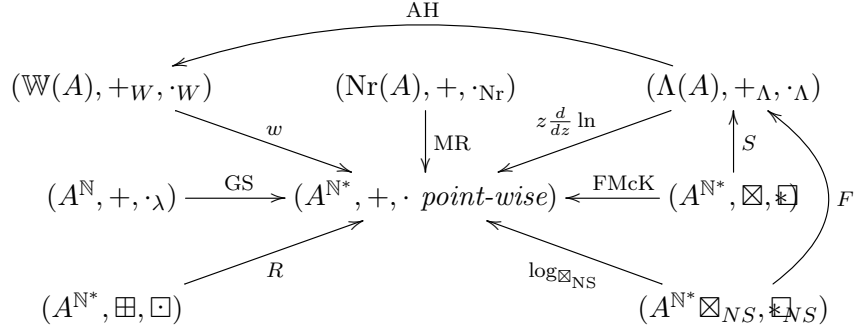
\[
\begin{xy}
  \xymatrix{
(\W(A),+_{W},\cdot_W)\ar[dr]^w&(\operatorname{Nr}(A),+,\cdot_{\operatorname{Nr}})\ar[d]^{\operatorname{MR}}&(\Lambda(A),+_{\Lambda},\cdot_{\Lambda})\ar[dl]_{z\frac{d}{dz}\ln}\ar@/_2pc/[ll]_{\operatorname{AH}} \\
(A^{\N},+,\cdot_{\lambda})\ar[r]^{\operatorname{GS}\qquad}&(A^{\N^*},\text{$+,\cdot$ point-wise})&(A^{\N^*},\boxtimes,\square\hspace{-0.85 em}\ast)\ar[l]_{\quad\qquad\operatorname{FMcK}}\ar[u]_S\\
(A^{\N^*},\boxplus,\boxdot)\ar[ur]_R&& (A^{\N^*}\boxtimes_{NS},\square\hspace{-0.85 em}\ast_{NS})\ar[ul]^{\log_{\boxtimes_{\operatorname{NS}}}}\ar@/_3pc/[uu] _F }
\end{xy}
\]
Here, $(\operatorname{Nr}(A),+,\cdot_{\operatorname{Nr}})$ is the {\em algebra of necklaces} of Metropolis and Rota~\cite{MR} and the labeled arrows indicate the isomorphism which linearises the ring structure. 
\end{thm}

Let us introduce the functor $\mathbb{V}_{\boxplus}:\mathbf{cAlg}_k\rightarrow\mathbf{cAlg}_k$ as 
$\mathbb{V}_{\boxplus}(A):=(A^{\N^*},\boxplus,\boxdot)$. 
Then the $R$-transform induces on $(A^{\N^*},\boxplus,\boxdot)$ a comonad structure, and similarly the $S$-transform. 
\begin{thm} 
There is a comonad $(\mathbb{V}_{\boxplus},\Delta, \varepsilon)$ on $\mathbf{cAlg}_k$, with the co-multiplication $\Delta:\mathbb{V}_{\boxplus}\rightarrow\mathbb{V}_{\boxplus}^2$ defined as shown in the commutative diagram below. For $A\in\mathbf{cAlg}_k$:
 \[
\begin{xy}
  \xymatrix{
\mathbb{V}_{\boxplus}(A)\ar[d]_{R_A}\ar[rrr]^{\Delta_A}&&&\mathbb{V}^2_{\boxplus}(A)\\
A^{\N^*}\ar[r]_{\delta_A}&A^{\N^*\times\N^*}\ar[rr]^{\sim}_{\delta(R^{-1}_A)}&&\mathbb{V}_{\boxplus}(A)^{\N^*}\ar[u]_{R^{-1}_{\mathbb{V}_{\boxplus}(A)}}
     }
\end{xy}
\]
and $\varepsilon:=\varepsilon^{\sigma}\circ R$, where $\delta$ is given in~(\ref{delta_sigma}).
\end{thm}

Analogously to Witt vectors, cf,~\cite{Hes}, we have ``{\bf Adams operations}".  
Let $(A,f)$ be a co-algebra over the comonad $(\mathbb{A}^{\N^*},\Delta,\varepsilon)$.  Then one can define operations $\psi^n$, $n\in\N^*$,  by
\[
\begin{xy}
  \xymatrix{
\psi^n:A\ar[r]^{f}&\mathbb{A}^{\N^*}(A)\ar[r]^{R_n}&A,
     }
\end{xy}
\]
where $R_n:=\varepsilon(\partial^{n-1}R(\underline{a}))$, i.e. $R_n$ is the projection onto the $n$th component of the $R$-transform.

\section{Analytic Aspects of Infinitely Divisible Probability Measures}
Here we generalise and extend the analytic results originally stated in~\cite{FMcK2012,FMcK2013a}.
\subsection{The Borel functor and free harmonic analysis}
Let $M_c:\mathbf{Top}_{\operatorname{lcH}}\rightarrow\mathbf{Set}$ denote the functor from the category of locally compact Hausdorff topological spaces to the category of sets, which assigns to every $T\in\mathbf{Top}_{\operatorname{lcH}}$ the set of compactly supported Borel probability measures on $T$. We note that $\mathbf{Top}_{\operatorname{lcH}}$ is dual to the category of commutative $C^*$-algebras. 

So, $M_c(\C)$ corresponds to the set of all compactly supported Borel probability measures on $\C$, and similarly $M_c(\R), M_c(\R^*_+)$ and $M_c(S^1)$ to the subsets of compactly supported probability measures on $\R$, $\R^*_+:=(0,\infty)$ and the unit circle $S^1$, respectively. Finally, let $\operatorname{Meas_c}$ denote the functor from $\mathbf{Top}_{\operatorname{lcH}}$ to compactly supported finite Borel measures.

Let $M^*(-)$ be the set of $\mu\in M_c(-)$ with non-vanishing first moment. Finally, let $M_{\infty,\bullet}(-)\subset M_c(-)$ be the set of all $\bullet$-infinitely divisible probability measures, where $\bullet$ is either $\boxplus$ or $\boxtimes$.

If  $a$ is a self-adjoint element in a $C^*$-probability space, and hence satisfies $a=a^*$,  then there exists a unique, compactly supported probability measure $\mu_a\in M_c(\R)$ such that for all $n\in\N$,
$$
\int_{\R} x^n d\mu_a(x)=\phi(a^n)
$$
holds.

We recall the content of Theorems~$3.7.2$, $3.7.3$ and $3.7.4.$ in~\cite{VDN}, which have been originally derived in~\cite{V1986} and~\cite{BV}. Let $\H$ denote the upper half-plane, $\overline{\H}$ the closed upper half-plane and $-\overline{\H}$ the closed lower half-plane, respectively. Further, $\Re(z)$ and $\Im(z)$ denote the real and imaginary part of a complex number $z$.

\begin{thm}[\cite{BV,V1986,VDN}]
\label{AnalyticChar}
The following analytic characterisations of freely infinitely divisible probability measures with compact support in terms of their $R$- and $S$-transforms, respectively, hold:
\begin{enumerate}
\item Let $\mathcal{R}:=\mathcal{R}_{\infty,\boxplus}(\R)$ denote the set of stalks of functions, $R(z)$, which are analytic in a neighbourhood of $(\C\setminus\R)\cup\{0\}$ (direct limit), and satisfy for all $z\in\H$: 
\begin{equation}
\label{R_upper_half_plane}
R(\bar{z})=\overline{R(z)}\quad\text{and}\quad \Im(R(z))\geq0.
\end{equation}
Then there exists a $\mu\in M_{\infty,\boxplus}({\R})$ such that $R=R_{\mu}$. Further, if $\mu\in M_{\infty,\boxplus}({\R})$ then the germ $R_{\mu}(z)$ can be analytically continued to $\C\setminus\R$ and to a neighbourhood of $0$ and the condition~(\ref{R_upper_half_plane}) holds.
\item Let $\mathcal{V}:=\mathcal{V}_{\infty,\boxtimes}(\R^*_+)$ denote the set of stalks of functions, $v(z)$, which are analytic in a neighbourhood of $(\C\setminus\R)\cup[-1,0]$ (direct limit), and satisfy for all $z\in\H$: 
$$
v(\bar{z})=\overline{v(z)}\quad\text{and}\quad \Im(v(z))\leq0.
$$
Then there exists a $\boxtimes$-infinitely divisible probability measure $\mu\in M_c({\R^*_+})$ such that the function $S(z):=\exp(v(z))$ satisfies $S=S_{\mu}$.
\item Let $\mathcal{U}:=\mathcal{U}_{\infty,\boxtimes}(S^1)$ denote the sheaf of functions, $u(z)$, which are analytic in ${H}_{-\frac{1}{2}}:=\{z\in\C~|\Re(z)>-\frac{1}{2}\}$ and satisfy for all $z\in H_{-1/2}$:
$$
\Re(u(z))\geq 0.
$$
Then there exists a  $\boxtimes$-infinitely divisible measure $\mu\in M^*_{\infty,\boxtimes}(S^1):= M_c({S^1})\cap M_c^{*}(\C)$ such that the function $S(z):=\exp(u(z))$ satisfies $S=S_{\mu}$.
\end{enumerate}
\end{thm}

\begin{lem}
\label{convex_cones}
The following holds:
\begin{enumerate}
\item  $(\mathcal{R},+)$, $(\mathcal{V},+)$ and $(\mathcal{U},+)$ are {\em commutative monoids} (semi-groups), with neutral element $0$. 
\item The sets $\mathcal{R}$, $\mathcal{V}$ and $\mathcal{U}$ are {\em convex cones}, i.e.  if  $f,g\in \mathcal{R}, \mathcal{V}, \mathcal{U}$ and $\alpha,\beta\in\R_+$, then $\alpha f+\beta g\in\mathcal{R}, \mathcal{V}, \mathcal{U}$. 
\end{enumerate}
\end{lem}
\begin{proof}
We only have to show the second statement, as the first one follows from $\alpha=\beta=1$ and $\alpha=0$ or $\beta=0$, respectively. 

Let us consider the situation for $\mathcal{R}$. If $f(z),g(z)\in\mathcal{R}$ and $\alpha,\beta\geq0$ then $\alpha f(\overline{z})=\overline{\alpha f(z)}$ and $\Im(\alpha f(z))=\alpha\Im(f(z))\geq0$ for $z\in\H$. Further, $\alpha f(\overline{z})+\beta g(\overline{z})=\overline{\alpha f(z)}+\overline{\beta g(z)}$ and $\Im(\alpha f(z)+\beta g(z))=\alpha\Im(f(z))+\beta\Im(g(z))\geq0$.

The remaining cases are treated similarly. 
\end{proof}

\begin{prop}
$(M_{\infty,\boxplus}(\R),\boxplus)$ is a commutative monoid with $\delta_0$ as neutral element. It is a sub-monoid of $(M_c(\R),\boxplus)$ and the $R$-transform, yields an isomorphism of commutative monoids:
$$
R:(M_{\infty,\boxplus}(\R),\boxplus)\rightarrow(\mathcal{R}_{\infty,\boxplus}(\R),+).
$$ 

\end{prop}
\begin{proof}
The first statement follows from 
$$
\mu\boxplus\nu=\mu_n^{\boxplus n}\boxplus\nu_n^{\boxplus n}=(\underbrace{\mu_n\boxplus\nu_n}_{\in M_c(\R)})^{\boxplus n}
$$
for some $\mu_n,\nu_n\in M_c(\R)$, $n\in\N^*$.
\end{proof}
\begin{prop}[Holomorphic linearisation of $\boxtimes$]
We have:
\begin{enumerate}
\item
There exists an isomorphism of semi-groups, 
$$
\operatorname{EXP}:(\mathcal{V},+)\rightarrow (M_{\infty,\boxtimes}(\R_+^*),\boxtimes),
$$
given by
$v(z)\mapsto S^{-1}(e^{v(z)})$, which linearises $\boxtimes|_{\R^*_+}$.

\item
There exists an isomorphism of semi-groups, 
$$
\operatorname{EXP}:(\mathcal{U},+)\rightarrow (M^*_{\infty,\boxtimes}(S^1),\boxtimes),
$$
given by $u(z)\mapsto S^{-1}(e^{u(z)})$, which linearises $\boxtimes|_{S^1\cap{M}^*(\C)}$.

\end{enumerate}
\end{prop}
\begin{proof}

$$
S_{\mu\boxtimes\nu}(z)=S_{\mu}\cdot S_{\nu}=e^{v(z)}\cdot e^{u(z)}=e^{v(z)+u(z)}.
$$
\end{proof}

\begin{prop}
There exists a morphism of abelian semi-groups
$$
\operatorname{EXP}:(M_{\infty,\boxplus}(\R),\boxplus)\rightarrow (M^*_{\infty,\boxtimes}(S^1),\boxtimes),
$$
given by 
$
\mu\mapsto S^{-1}(e^{-iR_{\mu}(i(z+1/2))}).
$
\end{prop}
\begin{proof}
Let $\varphi:H_{-\frac{1}{2}}\rightarrow\H$ be the conformal map $z\mapsto\varphi(z):=i(z+\frac{1}{2})$.  Now, for $R_{\mu}(z)\in \mathcal{R}_{\infty,\boxplus}(\R)$, the map
\begin{equation}
R_{\mu}\mapsto (-i R_{\mu})\circ\varphi,
\end{equation} 
is an injective map $\mathcal{R}\rightarrow\mathcal{U}$, with the transformed function being equal to $z\mapsto -iR_{\mu}(i(z+1/2))$.

For $\delta_0$ we have $R_{\delta_0}(z)\equiv0$ which is mapped onto $1$ and corresponds to $\delta_{e^0}$.  
Next, for $\mu,\nu\in M_{\infty,\boxplus}(\R)$ we have $R_{\mu}+R_{\nu}=R_{\mu\boxplus\nu}$, and further
$$
(iR_{\mu}(i(z+\frac{1}{2})))+(iR_{\nu}(i(z+\frac{1}{2})))=i(R_{\mu}+R_{\nu})\circ\varphi(z)=i(R_{\mu\boxplus\nu})\circ\varphi(z),
$$
from which the claim follows.
\end{proof}
Let us consider the following examples: 
\begin{itemize}
\item For $R(z)=bz$, i.e. the semi-circular law of radius $2\sqrt{b}$ and centred at the origin, we get $z\mapsto e^{bz+\frac{b}{2}}$, i.e. the {\bf free Brownian motion}.
\item For $R(z)=a$, i.e. the {\bf Dirac measure} $\delta_a$ at $a$, we get 
$a\mapsto e^{-ia}$.
\end{itemize}

The {\bf free Poisson distribution (fP) with rate $\lambda\geq0$ and jump size $\alpha\in\R$}, cf. e.g.~[\cite{NS}, Prop. 12.11, or \cite{VDN} p. 34],  is the limit in distribution for $N\rightarrow\infty$ of 
$$ 
\nu_{N,\lambda,\alpha}:=\left(\left(1-\frac{\lambda}{N}\right)\delta_0+\frac{\lambda}{N}\delta_{\alpha}\right)^{\boxplus N},
$$ 
with $\boxplus N$ in the exponent denoting the $N$-fold free additive self-convolution.

The ${R}$-transform of the limit $\nu_{\infty,\lambda,\alpha}:=\lim_{N\to\infty}\nu_{N,\lambda,\alpha}$ is
\begin{equation}
\label{inf_Poisson}
{R}_{\nu_{\infty,\lambda,\alpha}}(z)=\lambda\alpha\frac{1}{1-\alpha z}=\sum_{n=0}^{\infty}\lambda\alpha^{n+1}z^n=\lambda\alpha+\lambda\alpha^2 z+\lambda\alpha^3z^2+\cdots.
\end{equation}

Let $\mathcal{R}_{[-1,0]}(\R)\subset\mathcal{R}_{\infty,\boxplus}(\R)$ be the subset of stalks of $R$-transforms which are also analytic around $[-1,0]$ (direct limit set). Then $\mathcal{R}_{[-1,0]}(\R)$ is not empty, as it contains, e.g. the
\begin{itemize}
\item Dirac delta $\delta_a$,  $a\in\R$,
\item (Centred) semi-circular $a+bz$, $a\in\R$, $b>0$,
\item Free Poisson $\lambda\alpha/(1-\alpha z)$, $\lambda\in\R_+$ and $|\alpha|<1$.
\end{itemize}

\begin{prop}
There exists a monomorphism of abelian semi-groups 
$$
\operatorname{EXP}: (M_{\infty,\boxplus}(\R)|_{\mathcal{R}_{[-1,0]}},\boxplus)\rightarrow (M_{\infty,\boxtimes}({\R^*_+}),\boxtimes),
$$
given by $\mu\mapsto S^{-1}(e^{-R_{\mu}(z)})$.

\end{prop}
\begin{proof}
For $\delta_0$ we have $R_{\delta_0}(z)\equiv 0$ which maps to $e^0=1$, corresponding to $S_{\delta_{e^0}}(z)$.
Define a map $\mathcal{R}_{[-1,0]}\rightarrow\mathcal{V}$ by $R_{\mu}\mapsto-R_{\mu}$, which is injective. Then, for $v(z):=-R_{\mu}(z)$, we have
$$
v(\bar{z})=-R_{\mu}(\bar{z})\underbrace{=}_{\text{by $1.$}}-\overline{R_{\mu}(z)}=\overline{-R_{\mu}(z)}=\overline{v(z)},
$$ 
and for $\Im(z)>0$ we have $\Im(-R(z))=-\Im(R(z))\leq0$. Further,  $\exp(-(R_{\mu}+R_{\nu}))=\exp(-R_{\mu\boxplus\nu})=\exp(-R_{\mu})\cdot\exp(-R_{\nu})$ holds.
\end{proof}
\subsection{Analytic Witt semi-ring}
A complex sequence $(s_1,s_2,s_3,\dots)$ is called {\bf conditionally positive definite} if the shifted sequence $(s_2,s_3,\dots)$ is positive definite, cf.~[\cite{NS}, Notation 13.10]. If $(s_0,s_1,s_2,\dots)$ is a positive definite sequence then the truncated sequence $(s_1,s_2,s_3,\dots)$ is conditionally positive definite. 
Namely, for all $n\in\N$, and all $a_0,\dots, a_n\in\C$, we have, by assumption, that
$$
\sum_{i,j=0}^n a_i\bar{a}_j s_{i+j}\geq0.
$$
Let $a_0\equiv0$. Then for $n\geq1$ and for all $a_1,\dots,a_n\in\C$, we have 
$$
0\leq\sum_{i,j=0}^n a_i\bar{a}_j s_{i+j}=\sum_{i,j=1}^n a_i\bar{a}_j s_{i+j},
$$
as we are restricting to the sub-matrix where one deletes the first row and first column. This is a particular case of the general fact that any principal sub-matrix of a positive definite matrix is again positive definite.

\begin{prop}
\label{Char_infinitely_div}
Let $({s_n})_{n\in\N^*}$, $s_n\in\R$, be a sequence of real numbers. Then the following  statements are equivalent.
\begin{enumerate}
\item The sequence $({s_n})_{n\in\N^*}$ is conditionally positive definite and exponentially bounded, i.e. there exists a $C>0$ such that $|s_n|\leq C^n$ for all $n\in\N^*$.
\item There exists a $\mu\in M_{\infty,\boxplus}(\R)$ such that $\kappa_n(\mu)=s_n$, $n\in\N^*$, i.e. the real sequence is given by the free cumulants of a freely, infinitely divisible and compactly supported probability measure.
\item There exists a unique pair $(\gamma,\rho)$ with $\gamma\in\R$ and $\rho\in M_c(\R)$ such that
$$
\sum_{n=0}^{\infty} s_{n+1}z^n=\gamma+\int_{\R}\frac{z}{1-xz}d\rho(x).
$$
\end{enumerate}
\end{prop}

\begin{lem}
\label{shifted_cummulants}
Let $\mu\in M_{\infty,\boxplus}(\R)$ and $\kappa({\mu})=(\kappa_n(\mu))_{n\in\N^*}$ its free cumulant sequence. Then there exists a $\nu\in M_{\infty,\boxplus}(\R)$ such that 
$
\kappa_n({\nu})=\kappa_{n+2}({\mu}),
$ for all $n\in\N^*$.
\end{lem}
\begin{proof}[Proof of Lemma~\ref{shifted_cummulants}]
Define a sequence $(s_n)_{n\in\N^*}$ with $s_n:=\kappa({\mu})_{n+2}$ for $n\in\N^*$ which is conditionally positive definite and exponentially bounded. 
Then the claim follows from Proposition~\ref{Char_infinitely_div}.  
\end{proof}
\begin{proof}[Proof of Proposition~\ref{Char_infinitely_div}]
$2.\Rightarrow 1.:$ This follows from~[\cite{NS} Theorem 13.6 and Proposition 13.15].

$1.\Rightarrow 3.:$ This follows from~[\cite{NS} Proposition 13.14].

$3.\Rightarrow 1.:$ This follows from~[\cite{NS} Lemmas 13.13 and 13.14].

$1.\Rightarrow 2.:$
By~[\cite{NS} Proposition 13.14], there exists a finite, compactly supported, Borel measure $\rho$ on $\R$ such that 
\begin{equation}
\label{generating_function}
R(z):=\sum_{n=0}^{\infty} s_{n+1}z^n=s_1+\int_{\R}\frac{z}{1-xz}d\rho(x).  
\end{equation}

\item We show that $R(z)$ in~(\ref{generating_function}) satisfies [Theorem~\ref{AnalyticChar},~1.].  
\begin{itemize}
\item
By using a version of the ``Differentiation Lemma", cf. e.g.~[\cite{Kle} Theorem 6.28], it follows that $R(z)$ is analytic in $\H$ and $-\H$ and in a neighbourhood of $0$. The statement for the first two regions follows directly from the integral representation. 

As the support of $\rho$ is bounded, there exists an $A>0$ such that $\supp(\rho)\subset[-A,A]$. Let $|z|<C_{\rho}:=\min(1,1/A)$. Then $|zx|<1$ for all $x\in\supp(\rho)$, and hence $R(z)$ is analytic in a disc of radius $C_{\rho}$ centred at the origin.
\item Positive imaginary part. We split the integral into real and imaginary parts. For $z:=a+ib$, we have  
$$
\frac{z}{1-xz}=\underbrace{\frac{a-x(a^2+b^2)}{(1-ax)^2+x^2b^2}}_{=:u(x)}+i\underbrace{\frac{b}{(1-ax)^2+x^2b^2}}_{=:v(x)}=u(x)+iv(x).
$$
Then 
$$
R(\overline{z})=s_1+\int_{\R} u(x)d\rho(x)-i\int_{\R}v(x)d\rho(x)=\overline{R(z)}
$$
\item For all $x\in\R$ and $b>0$, we have $v(x)>0$, and so for $z\in\H$,
$$
\Im(R(z))=\int_{\R}v(x)d\rho(x)\geq0.
$$
\end{itemize}
Hence, there exists a $\mu\in M_{\infty,\boxplus}(\R)$ with $R_{\mu}(z)=R(z)$.
\end{proof}

Let us give a second proof, which uses a Fock space construction. 
\begin{proof}[Second proof of Proposition~\ref{Char_infinitely_div}]

There exists a finite, compactly supported measure $\rho$ on $\R$, whose moments satisfy $m_n(\rho)=s_{n+2}$, for all $n\in\N$.

Let $\mathfrak{m}$ be the unique maximal ideal of $\C[x]$. Define an inner product $\langle~,~\rangle$ on $\mathfrak{m}$, first on basis vectors by
$$
\langle x^m,x^n\rangle:=s_{m+n},\quad m,n\in\N^*,
$$
and then extend it sesquilinearly to arbitrary polynomials. Let   
$$
\mathcal{N}:=\{p(x)\in\mathfrak{m}|\langle p(x),p(x)\rangle=0\},
$$
be the corresponding null space. The Hilbert space is then given by
$$
\mathcal{H}:=\overline{\mathfrak{m}\big/\mathcal{N}},
$$
where the completion is taken with respect to $\langle~,~\rangle$.

Let $T(\mathcal{H})$ be the associated tensor algebra with $\mathfrak{B}(\mathcal{H})$ and $\mathfrak{B}(T(\mathcal{H}))$, the vector spaces of bounded linear operators on $\mathcal{H}$ and $T(\mathcal{H})$, respectively.

Consider the $C^*$-probability space $(\mathfrak{B}(T(\mathcal{H})),\tau_{\mathcal{H}})$ and the self-adjoint linear operator $\mathfrak{a}\in\mathfrak{B}(T(\mathcal{H}))$, given by
$$
\mathfrak{a}:=\left(l(x)+l^*(x)+\Lambda(x)+s_1\id_{T(\mathcal{H})}\right),
$$
where $l(x)$ and $l^*(x)$ are the raising and lowering operators, respectively. Further, let $T_x$ be the multiplication operator on $\mathcal{H}$, i.e. $T_x(p(x)):=x\cdot p(x)$ for all $p(x)\in\mathfrak{m}$, and define $\Lambda(x):=\Lambda(T_x)$, as the gauge operator corresponding to $T_x$. 

The operator $\mathfrak{a}$ is bounded as $\|l(x)\|=\|l^*(x)\|=\|x\|=\sqrt{s_{2}}$, $\|s_1\cdot\id\|=|s_1|$ are so, and because $\Lambda(x)$ is also bounded as the support of $\rho$ is compact. 

By~[\cite{NS} Proposition~13.8] the distribution of $\mathfrak{a}$ is infinitely divisible, i.e. there exists a $\mu_{\mathfrak{a}}\in M_{\infty,\boxplus}(\R)$ such that
$$
m_n(\mathfrak{a})=\tau_{\mathcal{H}}(\mathfrak{a}^n)=\int_{\R}x^nd\mu_{\mathfrak{a}}=m_n(\mu_{\mathfrak{a}}),
$$
for all $n\in\N$.
From [\cite{NS} Proposition 13.5] and as in [\cite{NS} Theorem 13.16, p.233] we have
\begin{eqnarray*}
\kappa_1(\mu_{\mathfrak{a}})=\kappa_1(\mathfrak{a})&=&s_1\\
\kappa_2(\mu_{\mathfrak{a}})=\kappa_2(\mathfrak{a})&=&\kappa_2(l^*(x)l(x))=\langle x,x\rangle=s_2,
\end{eqnarray*}
and for $n\geq2$,
\begin{equation*}
\kappa_n(\mu_{\mathfrak{a}})=\kappa_n(\mathfrak{a})=\kappa_n(l^*(x),\underbrace{\Lambda(x),\dots,\Lambda(x)}_{\text{$(n-2)$-times}},l(x))=\langle x,\Lambda(x)^{n-2} x\rangle=\langle x,x^{n-1}\rangle=s_{n}.
\end{equation*}
As all the free cumulants of $\mu_{\mathfrak{a}}$ agree with the sequence $(s_n)_{n\in\N}$,  $\mu_{\mathfrak{a}}$ is the wanted measure. 

Let us show that $T_x$ is bounded for $\rho$ with compact support. 
For $\xi(x), f(x)\in L^2(\R,\rho)$ with $|f|\big|_{\supp(\rho)}\leq C_f$, $C_f>0$, we have
$$
\|\Lambda(f)(\xi)\|=\|f\cdot\xi\|\leq C_f\|\xi\|
$$
as
$$
\|f\xi\|^2=\int_{\R}|f(x)|^2|\xi(x)|^2 d\rho(x)\leq\int_{\R} C_f^2|\xi(x)|^2 d\rho\leq C_f^2\|\xi\|^2.
$$
Then for $m\in\N^*$, we have
\begin{eqnarray*}
\|\Lambda(x)(x^m)\|^2=\langle\Lambda(x)(x^m),\Lambda(x)(x^m)\rangle=\langle x^{m+1}, x^{m+1}\rangle=s_{2m+2}\\
\int_{\R}x^{2m}d\rho(x)=\int_{\R}x^2\cdot x^{2m-2}d\rho(x)\leq C_{x}\cdot\|x^{m-1}\|\leq C_{x}\cdot\langle x^m,x^m\rangle=C_{x}\|x^m\|^2.
\end{eqnarray*}
\end{proof}

For two power series, the point-wise or Hadamard multiplication, $\cdot_H$, is defined as:
$$
\left(\sum_{n=0}^{\infty} a_n z^n\right)\cdot_H \left(\sum_{n=0}^{\infty} b_n z^n\right):=\sum_{n=0}^{\infty} a_n  b_b z^n~.
$$

The $R$-transform is a monomorphism $R:M(\R)\rightarrow \R[[z]]$ which is given by $\mu\mapsto R_{\mu}(z)$. For $\mu,\nu\in\M_c(\R)$, we define a binary operation $\boxdot$, as follows: 
\begin{equation}
\label{boxdot}
\mu\boxdot\nu:=R^{-1}(R_{\mu}(z)\cdot_H R_{\nu}(z)),
\end{equation}
where $R^{-1}$ is the inverse of the $R$-transform. However, sometimes it is preferable to use $\mathcal{R}:=zR$.  
\begin{thm}
$(M_{\infty,\boxplus}(\R),\boxplus,\boxdot,\delta_0, \nu_{\infty,1,1})$ is a commutative semi-ring with the free Poisson distribution $\nu_{\infty,1,1}$ as multiplicative unit.  
\end{thm}

\begin{proof}
We have $\nu_{\infty,1,1}\in\M_{\infty,\boxplus}(\R)$ which follows from, e.g. the characterisation in Theorem~\ref{AnalyticChar}~1, and further $R_{\nu_{\infty,1,1}}=(1,1,1,\dots)$, which is clearly the unit for $\boxdot$.

We have to show that this product is well defined.
\begin{enumerate}
\item
Let $\mu,\nu\in\M_{\infty,\boxplus}(\R)$ and  consider the corresponding free cumulant sequences $(k_n(\mu))_{n\in\N}$ and $(k_n(\nu))_{n\in\N}$. 
Both sequences are conditionally positive definite and by~[\cite{NS} Theorem 13.16, properties (1) and (2)], [\cite{NS} Theorem 13.16, property (3)] and [\cite{NS} Lemma 13.13] also exponentially bounded, i.e. $|\kappa_n(\bullet)|\leq C_{\bullet}^n$, $n\in\N$ for some constant $C_{\bullet}>0$, depending on the measures $\mu$ and $\nu$, respectively.

Hence, $|\kappa_n(\mu)\cdot\kappa_n(\nu)|\leq (C_{\mu}\cdot C_{\nu})^n$ for all $n\in\N$, i.e. the point-wise product of the free cumulants does not grow faster than exponentially. 

\item
In order to show positivity, we consider the corresponding Hankel matrices. Let $\tilde{\kappa}_n(\bullet):=\kappa_{n+2}(\bullet)$ for $n\in\N$. Then, by the Schur Product Theorem, 
$$
(\tilde{\kappa}_{i+j}(\mu)\cdot \tilde{\kappa}_{i+j}(\nu))_{0\leq i,j\leq n},
$$ 
is positive definite.  

Therefore the sequence $(h_n)_{n\in\N^*}$, with $h_1:=\kappa_1(\mu)\cdot\kappa_1(\nu)$ and $h_n:=\kappa_n(\mu)\cdot\kappa_n(\nu)$, $n\in\N^*$, is conditionally positive definite and exponentially bounded.

By Proposition~\ref{Char_infinitely_div} there exists 
a probability measure $\xi\in M_{\infty,\boxplus}(\R)$ such that $h_n=\kappa_n(\xi)$,$n\in\N^*$. Finally, let $\mu\boxdot\nu:=\xi$.
\item
Associativity and distributivity follow from the fact that the operations are given component-wise, which individually satisfy these properties. Apply $R$ to 
$(\mu\boxplus\nu)\boxdot\eta$, which by definition, yields, $R_{\mu\boxplus\nu}\cdot_H R_{\eta}=(R_{\mu}+R_{\nu})\cdot_H R_{\eta}$ from which the result $\mu\boxdot\eta\boxplus\nu\boxdot\eta$, after applying $R^{-1}$, follows.
\end{enumerate}
\end{proof}

We remark that the binary operation $\boxdot$ can be lifted via the exponential map to $M_{\infty,\boxtimes}^*(S^1)$ and $M_{\infty,\boxtimes}^*(\R^*_+)$, respectively. However, in the analytic case some additional considerations are necessary.
\subsection{Geometric properties}
\begin{prop}
There exists a surjection $\pi:M_{\infty,\boxplus}(\R)\rightarrow \operatorname{Meas}_c(\R)$, with fibre $\pi^{-1}(\rho)\cong\R\times\{\rho\}\cong\R$.
\end{prop}
\begin{prop}
The space $M_{\infty,\boxplus}(\R)$ carries two actions: for $\mu\in M_{\infty,\boxplus}(\R)$, $r\in\R$ and $c\in\R_+^*$, we have:
\begin{itemize}
\item the $\R$-action,(shift along the fibre) is given by:
$$
r.\mu:=\delta_{r}\boxplus\mu,
$$
\item
the $\R^*_+$-action (scaling) is given by: 
$$
c.\mu:=\nu_{\infty,c,1}\boxdot \mu.
$$
\end{itemize}
\end{prop}
\begin{proof}
$(r+r').\mu=\delta_{r+r'}\boxplus\mu=\delta_r\boxplus(\delta_{r'}\boxplus\mu)$ and $0.\mu=\delta_0\boxplus\mu=\mu$.
\end{proof}

Example: for the Dirac, semi-circular and free Poisson distribution, with $c,r,\lambda\geq0$, and $\alpha\in\R$, we have
\begin{eqnarray*}
c.\delta_a&=& \delta_{c\cdot a},\\
c.\gamma_{a,r}&=& \gamma_{c\cdot a,{\sqrt{c}}\cdot r},\\
c.\nu_{\lambda,\alpha}&=& \nu_{c\cdot\lambda,\alpha},
\end{eqnarray*}
which shows that these classes of measures are invariant under the action. In particular, the semi-circular distributions form a convex sub-cone.

\subsection{Endomorphisms}

\begin{prop}
The family of Dirac distributions is a two-sided ideal in $M_{\infty,\boxplus}(\R)$. We have  
\begin{eqnarray*}
\delta_a\boxplus\delta_b&=&\delta_{a+b},\\
\delta_a\boxdot\delta_b&=&\delta_{ab}.
\end{eqnarray*}
Further for $\mu\in M_{\infty,\boxplus}(\R)$ we have
\begin{equation*}
\delta_a\boxdot\mu=\delta_{\kappa_1(\mu)\cdot a}
\end{equation*}
where $\kappa_1(\mu)$ is the first free cumulant of $\mu$.
\end{prop}
For $a,b\in\R$ and $r,s>0$, the {\bf semicircle law} centred at $a$ and of radius $r$, is defined as the distribution $\gamma_{a,r}:\C[z]\rightarrow\C$, given by
\begin{equation*}
\gamma_{a,r}(z^n):=\frac{2}{\pi r^2}\int_{a-r}^{a+r} t^n\sqrt{r^2-(t-a)^2}\,dt\qquad \forall n\in\N,
\end{equation*}
and its ${R}$-transform is 
\begin{equation}
\label{Rsemicircle}
{R}_{\gamma_{a,r}}(z)=a+\frac{r^2}{4}z~.
\end{equation}
\begin{prop}
Let $\gamma_{a,0}:=\delta_a$. The family of semicircular distributions is a two-sided ideal in $M_{\infty,\boxplus}(\R)$, which contains the Dirac distributions as sub-ideal. We have 
\begin{eqnarray}\nonumber
\gamma_{a,r}\boxplus\gamma_{b,s} & = & \gamma_{a+b,\sqrt{r^2+s^2}}, \\
\gamma_{a,r}\boxdot\gamma_{b,s} & = &  \gamma_{ab, rs/2}.
\end{eqnarray}  
\end{prop} 
\begin{proof}
We have ${R}_{\gamma_{a,r}\boxdot\gamma_{b,s}}(z)=abz+\frac{r^2 s^2}{4\cdot 4}z^2$ by component-wise multiplication of the ${R}$-transforms and the relation~(\ref{Rsemicircle}).

For any $\mu\in M$, $\kappa(\mu)=(\kappa_1(\mu),\kappa_2(\mu),\dots)$ with $\kappa_1(\mu)\in\R$ and $\kappa_2(\mu)\geq0$ as, e.g. by~[\cite{NS} Prop], $\kappa_2(\mu)=m_0(\rho)=\rho(\R)\geq0$ for some compactly supported finite measure $\rho$. Therefore, the non-negativity of the second free cumulant is always preserved under the commutative point-wise multiplication, and therefore the property of being an ideal follows. 
\end{proof}
\begin{prop}
The family of free Poisson distributions $\mathbf{fP}$ with rate $\lambda\in\R_+$ and jump size $\alpha\in\R$ is closed with respect to the $\boxdot$-convolution, i.e. for $\lambda,\lambda'\in\R_+$ and $\alpha,\beta\in\R$ we have
$$
\nu_{\infty,\lambda,\alpha}\boxdot\nu_{\infty,\lambda',\beta}=\nu_{\infty,\lambda\lambda',\alpha\beta}.
$$
Therefore, $\mathbf{fP}$ is a $\boxdot$-monoid, and for $(\lambda,\alpha)\in\R^*_+\times\R^*$, an abelian group. 
\end{prop}
\begin{proof}
The statement follows from the representation given in equation~(\ref{inf_Poisson}) by component-wise multiplication.
\end{proof}

Let us remark that a more general statement should be valid which holds also for the {\bf compound free Poisson distribution}, cf.~[\cite{NS}~p.206]. 

We derived our analytic results for compactly supported probability measures, but it is natural to assume that one can extend them to all freely infinitely divisible probability measures.

For $a\in\R$, we define the {\bf Teichmüller representative} as $\tau:\R\rightarrow M_{\infty,\boxplus}(\R)$, $a\mapsto\nu_{\infty,1,a}$ which gives a group isomorphism $\tau:(\R^*,\cdot)\rightarrow (\nu_{\infty,1,\bullet},\boxdot)$.

Let us define the shift operator $\mathbf{V}$.  For $\mu\in M_{\infty,\boxplus}(\R)$, with free cumulants $\kappa({\mu})=(\kappa_n({\mu}))_{n\in\N^*}$, let $\mathbf{V}(\mu)\in M_{\infty,\boxplus}(\R)$, such that for all $n\in\N^*$, we have:
$$
\kappa_n(\mathbf{V}(\mu)):=\kappa_{n+2}(\mu),
$$
which by Lemma~\ref{shifted_cummulants} is well-defined.
\begin{prop}
The operator $\mathbf{V}$ defines a semi-ring endomorphism of $M_{\infty,\boxplus}(\R)$, i.e. for $\mu_1,\mu_2\in M_{\infty,\boxplus}(\R)$, we have:
\begin{eqnarray*}
\mathbf{V}(\mu_1\boxplus\mu_2)&=&\mathbf{V}(\mu_1)\boxplus \mathbf{V}(\mu_2),\\
\mathbf{V}(\mu_1\boxdot\mu_2)&=&\mathbf{V}(\mu_1)\boxdot \mathbf{V}(\mu_2).
\end{eqnarray*}
In particular, it respects the additive and multiplicative unit, i.e. we have
\begin{eqnarray*}
\mathbf{V}(\delta_0)&=&\delta_0,\\
\mathbf{V}(\nu_{\infty,1,1})&=&\nu_{\infty,1,1}.
\end{eqnarray*}
\end{prop}
\begin{proof} 
From
$\kappa_n(\mu_1\boxplus\mu_2)=\kappa_n(\mu_1)+\kappa_n(\mu_2)$ and, by definition $\kappa_n(\mu_1\boxdot\mu_2)=\kappa_n(\mu_1)\cdot_{H}\kappa_n(\mu_2)$, $n\in\N^*$, the claim follows.
\end{proof}

For $a\in A^{\N^*}$, $n\in\N$, with $0^0:=1$, let $\mathbf{f}_n:A^{\N^*}\rightarrow A^{\N^*}$ be given by
\begin{equation}
\label{Frob_map}
(a_j)_{j\in\N^*}\mapsto (a_j^n)_{j\in\N^*}.
\end{equation}
At the level of measures we have, by slight abuse of notation, $\mathbf{f}_n(\mu):=R^{-1}(\mathbf{f}_n(\kappa(\mu)))$, which is in fact well defined.
\begin{prop}
For $\mu\in M_{\infty,\boxplus}(R)$, $n\in\N$, we have that:
\begin{equation}
\mathbf{f}_n(\mu)=\mu^{\boxdot n},
\end{equation}
and $\mathbf{f}_n$ yields a $\boxdot$-endomorphisms, i.e. for $\mu,\nu\in M_{\infty,\boxplus}(\R)$: 
\begin{eqnarray*}
\mathbf{f}_n(\mu\boxdot\nu)&=&\mathbf{f}_n(\mu)\boxdot\mathbf{f}_n(\nu),\\
(\mu\boxdot\nu)^{\boxdot n}&=&\mu^{\boxdot n}\boxdot\nu^{\boxdot n}.
\end{eqnarray*}
\end{prop}
\begin{proof}
$\mathbf{f}_n(\kappa(\mu))=(\kappa_j(\mu)^n)_{j\in\N^*}=\kappa(\mu)^{\boxdot n}$ from which the claim follows.
\end{proof}

We remark that the set of free Poisson distributions is invariant with respect to  $\mathbf{V}$ and $\mathbf{f}_n$, in particular 
$$
\mathbf{V}(\nu_{\infty,\lambda,\alpha})=\nu_{\infty,\lambda\alpha^2,\alpha}.
$$
\subsection{Relation with classical infinite divisibility}
The main references for this section are~\cite{BNT,BPB,C2014}. Here, we denote by $\operatorname{ID}(\R,*)$ and $\operatorname{ID}(\R,\boxplus)$ the sets of $*$-infinitely (classically) divisible and $\boxplus$-infinitely (freely) divisible probability measure on the real line; i.e. there is no restriction on the support.  

We have $M_{\infty,\boxplus}(\R)\subset\operatorname{ID}(\R,\boxplus)$ as a dense sub-monoid and that there is a bijection $\varphi$ between characteristic pairs $(\gamma,\sigma)$ and $\bullet$-infinitely divisible probability measures on $\R$, where $\bullet=*,\boxplus$, i.e.  we have
\begin{equation}
\label{LKpairs}
\varphi_{\operatorname{cp}}:\R\times\operatorname{Meas}(\R)\rightarrow\operatorname{ID}(\R,\bullet).
\end{equation}

Let $\mu\in\operatorname{ID}(\R,\boxplus)$ with characteristic pair $(\gamma,\sigma)$. Then from~[\cite{BNT} (4.3)] we obtain
\begin{equation}
\label{R-trafo_char_pair}
R_{\mu}(z)=\gamma+\int_{\R}\frac{z+x}{1-xz}d\sigma(x).
\end{equation}

\begin{prop}
Let $\pi_*:\operatorname{ID}(\R,*)\rightarrow\operatorname{Meas}(\R)$ and $\pi_{\boxplus}:\operatorname{ID}(\R,\boxplus)\rightarrow\operatorname{Meas}(\R)$ be the projections onto the second factor of the characteristic pair, i.e. $\pi_i(\mu):=\pr_2(\gamma,\sigma)=\sigma$, $i=*,\boxplus$ where $\mu$ is represented by $(\gamma,\sigma)$. The fibre $\pi^{-1}_i(\sigma)$ is isomorphic to $\R$, and the Berkovi-Pata bijection gives a bundle map $\varphi_{\operatorname{BP}}:\operatorname{ID}(\R,*)\rightarrow \operatorname{ID}(\R,\boxplus)$, i.e. the following diagram 
\[
\begin{xy}
  \xymatrix{
 \operatorname{ID}(\R,*)\ar[dr]_{\pi_*}\ar[rr]_{\varphi_{\operatorname{BP}}} & &\operatorname{ID}(\R,\boxplus)\ar[dl]^{\pi_{\boxplus}} \\
       & \operatorname{Meas}(\R)&
               }
\end{xy}
\]
commutes. 
\end{prop}

For $\mu,\nu\in\operatorname{ID}(\R,*)$ with classical cumulants $(c(\mu))_{n\in\N^*}$ and $(c(\nu))_{n\in\N^*}$, we define, equivalently to~(\ref{boxdot}), a binary operation $\star$ as follows:
\begin{equation}
\label{class_prod}
c(\mu\star\nu)_n:=c(\mu)_n\cdot c(\nu)_n,
\end{equation}
i.e. the two cumulant sequences are multiplied component-wise, and then the measure is obtained by the inverse Fourier transform of the exponential of the resulting sequence. 

We have:
\[
\begin{tabular}{l|r|l} distribution & parameters& classical cumulants   \\\hline 
Dirac & $a\in\R$ & (a,0,0,0,\dots) \\Normal & $m\in\R$, $\sigma^2\in\R^*_+$ & $(m,\sigma^2,0,0,\dots)$\\ 
Poisson & $\lambda\in\R_+^*$ & $(\lambda,\lambda,\lambda,\lambda,\dots)$ \end{tabular}
\]

\begin{prop}
The joint set of Dirac, normal and Poisson distributions form a commutative monoid for the operation $\star$, with unit the Poisson distribution with parameter $\lambda=1$.  
\end{prop}
\begin{rem}
Let us note that in equality~(\ref{generating_function}), the pair $(s_1,\rho)$ also uniquely characterises $\mu\in M_{\infty,\boxplus}(\R)$. From expression~(\ref{R-trafo_char_pair}) we obtain the relation $\kappa_1(\mu)=\gamma$, and for the measures $\sigma$ and $\rho$
$$
\int_{\R}\frac{z+x}{1-xz}d\sigma(x)=\int_{\R}\frac{z}{1-xz}d\rho(x).
$$
\end{rem}

\begin{prop}
The following statements hold:
\begin{enumerate}
\item
The sets $\operatorname{ID}(\R,\ast)$ and $\operatorname{ID}(\R,\boxplus)$ are convex cones for the following operation: for $\alpha,\beta\geq0$ and $\mu,\nu\in\operatorname{ID}(\R,\bullet)$ with $\bullet=*,\boxplus$,  let
$$
\alpha\mu+\beta\nu:=\varphi^{-1}(\alpha\gamma_{\mu}+\beta\gamma_{\nu},\alpha\mu+\beta\nu)
$$
where $\varphi_{\operatorname{cp}}^{-1}$ is the inverse of the bijection~(\ref{LKpairs}) .
\item The normal and the semicircular distributions form isomorphic convex sub-cones of $\operatorname{ID}(\R,\bullet)$.
\item
$\operatorname{ID}(\R,\bullet)$, $\bullet=*,\boxplus$, carries a fibre-wise $\R$-action, which is given by 
$$
\delta_a.\mu:=\varphi_{\operatorname{cp}}^{-1}((\gamma+a,\rho)),
$$
where $a\in\R$ and $\varphi_{\operatorname{cp}}^{-1}$ is the compositional inverse of the bijection(\ref{LKpairs}).
\end{enumerate}
\end{prop}
\section{Algebraic theories related to convolution of probability measures}
\subsection{Convex sets}

For $\mu,\nu\in M_c(\R)$ and $q\in[0,1]$ let $q\mu+(1-q)\nu$ be the usual {\bf $q$-convex linear combination} of measures. 
Then the moments satisfy:
$$
m_n(q\mu+(1-q)\nu)=qm_n(\mu)+(1-q) m_n(\nu),
$$ 
which follows from 
$$
\int_{\R} x^n d(q\mu+(1-q)\nu)=q\int_{\R}x^nd\mu+(1-q)\int_{\R}x^nd\nu.
$$
Hence the moment map $m$ is a morphism for convex linear combinations. 

Let us briefly recall the notion of convex sets, cf.~\cite{Fr,J}.
\begin{df}
\label{convex}
For every $p,q\in[0,1]$ we have a binary operation $+_q$ which satisfies the following rules:
\begin{enumerate}
\item $+_q(x,y)=+_{1-q}(y,x),$
\item $+_q(x,x)=x,$
\item $+_0(x,y)=y,$
\item $+_p(+_q(x,y),z)=+_{p+(1-p)q}(+_{\frac{p}{p+(1-p)q}}(x,y),z)$,\quad if~$p+(1-p)q\neq0$.
\end{enumerate}
\end{df} 

\begin{prop}
For $\mu,\nu\in M_{\infty,\boxplus}(\R)$ and $\alpha,\beta\in\R_+$, define the family of binary operations $+_{\alpha,\beta}$ as follows:
$$
\mu+_{\alpha,\beta}\nu:=R^{-1}(\alpha R_{\mu}+\beta R_{\nu}),
$$
which yields
$$
\mu+_{\alpha,\beta}\nu=\mu^{\boxplus\alpha}\boxplus\nu^{\boxplus\beta}.
$$
Therefore, $(M_{\infty,\boxplus}(\R),+_{\alpha,\beta})$ has the structure of a convex cone. 
\end{prop}
\begin{proof}
This follows from Lemma~\ref{convex_cones}. The second proof follows from the equality~(\ref{R-trafo_char_pair}).
\end{proof}
\begin{prop}
For $q\in[0,1]$, let $+_q:=+_{q,(1-q)}$. The operations $\boxplus$ and $\boxdot$ distribute over the addition $+_q$, i.e. for $\mu,\nu,\xi\in M_{\infty,\boxplus}(\R)$, we have:
\begin{eqnarray*}
(\mu+_q\nu)\boxplus\xi=(\mu\boxplus\xi)+_q(\nu\boxplus\xi),\\
(\mu+_q\nu)\boxdot\xi=(\mu\boxdot\xi)+_q(\nu\boxdot\xi).\\
\end{eqnarray*}
\end{prop}

\begin{proof}
In order to show distributivity, we calculate:
$$
R\left((\mu+_q\nu)\boxplus\xi\right)=(qR_{\mu}+(1-q)R_{\nu})+R_{\xi}=q(R_{\mu}+R_{\xi})+(1-q)(R_{\nu}+R_{\xi}),
$$
and
$$
R\left((\mu+_q\nu)\boxdot\xi\right)=(qR_{\mu}+(1-q)R_{\nu})\cdot_HR_{\xi}=q(R_{\mu}\cdot_HR_{\xi})+(1-q)(R_{\nu}\cdot_HR_{\xi}).
$$
\end{proof}

Let us point out that we have to distinguish between operations defined on moments and cumulants and in general we can not mix the two, as the following example shows, cf.~[\cite{NS} Remarks 12.10 (2)]. For $\mu:=\frac{1}{2}(\delta_{-1}+\delta_{+1})$, with the convex combination taken for moments, we have 
$$
\frac{1}{2}(\delta_{-1}+\delta_{+1})\boxplus\mu\neq\frac{1}{2}(\delta_{-1}\boxplus\mu)+\frac{1}{2}(\delta_{+1}\boxplus\mu).
$$
Let $S$ be a semi-ring and $G_S$ be the functor
$$
G_S:\mathbf{Set}\rightarrow\mathbf{Set}
$$
such that for $X\in\mathbf{Set}$, we have 
$$
G_S(X):=\left\{\varphi:X\rightarrow S|~\text{$\varphi(x)=0$ for almost all $x\in X$, and}\sum_{x\in X}\varphi(x)=1\right\}.
$$
Then every $\varphi\in G_S(X)$ has a representation
$$
\varphi=s_1x_1+\cdots+s_nx_n,
$$
with $s_i\in S$, $x_i\in X$, $i=1,\dots,n$, and where $s_i=0$ is possible and the $x_i$ are not necessarily different, i.e. we allow also non-reduced formal sums, which can be considered as special $S$-valued divisors. 
\begin{prop}[\cite{J,Fr}]
Let $S$ be a commutative unital semiring. The {\bf distribution} or {\bf finite Giry monad} $(G_S,\eta,\mu)$ with coefficients in $S$, is given by: 
\begin{eqnarray*}
\eta_X&:&X\rightarrow G_S(X),\\
x&\mapsto& 1_Sx,\\
\mu_X&:&G_S^2:=G_S\circ G_S\rightarrow G_S,\\
\mu_X\left(\sum_{i=1}^n s_i\varphi_i\right)&:=&x\mapsto\sum_{i=1}^n s_i\varphi_i(x)
\end{eqnarray*}
\end{prop}
\begin{prop}
$(M_{\infty,\boxplus}(\R),+_q, q\in[0,1])$ is an algebra for the finite Giry / distribution monad, with $\alpha:G_{\R_+}(M_{\infty,\boxplus}(\R))\rightarrow M_{\infty,\boxplus}(\R)$ given by:
$$ 
\alpha(\lambda_1\mu_1+\cdots\lambda_n x_n):=\mu_1+_{q_1}(\mu_2+_{\mu_2}(\cdots +_{q_{n-2}}(\mu_{n-1}+_{q_{n-1}}\mu_n)))
$$
for $\mu_i\in M_{\infty,\boxplus}(\R)$, $\lambda_i\in\R_+$, $\sum_{i=1}^n\lambda_i=1$, $q_1:=\lambda_1$ and $q_i:=\frac{q_i}{1-q_1-\dots-q_{i-1}}$, $i=1,\dots, n$.
\end{prop}
\begin{proof}
The statement follows from~[\cite{J} Theorem 4].
\end{proof}
\subsection{The convolution algebra }
We introduce, following the notation in~\cite{MacL}, an algebraic theory $(\Omega,E)$ consisting of a graded set $\Omega$ of operators and $E$ a set of identities which together  define an $\langle\Omega,E\rangle$-algebra.

\begin{df}
The {\bf (partial) convolution algebra} is given by the following generators $\Omega$ and relations $E$:
\begin{itemize}
\item $\Omega(0):=\{0,1\}\times\R$, nullary operations,
\item $\Omega(1):=\{\id,\operatorname{inv},\partial,f_n\}$, $n\in\N$, unary operations, 
\item $\Omega(2):=\{+,\cdot,+_q\}$, $q\in\Delta^1$, where $\Delta^1$ denotes the geometric one-simplex, binary operations.
\end{itemize}
These satisfy the following relations $E$:
\begin{enumerate}
\item $(+,\cdot,\mathbf{0}:=(0,0),\mathbf{1}:=(1,1))$ forms a commutative semi-ring with additive unit $\mathbf{0}$ and multiplicative unit $\mathbf{1}$,
\item $+_q$ $q\in[0,1]$ satisfies the properties of Definition~\ref{convex},
\item $\partial$ is an endomorphism for $+,\cdot$ and $+_q$, $q\in[0,1]$,
\item $f_n$ $n\in\N^*$ is an endomorphism for $\cdot$.
\end{enumerate}
\end{df}

\begin{rem} 
\begin{enumerate}
\item We call it ``partial" because there exist further operations, however we intended to capture those which we have encountered so-far.
\item
The above algebraic structure is not operadic, as its definition also involves co-operations. 
However, it contains several sub-operads, e.g. Definition~\ref{convex}, without property $2.$ defines an operad, cf. the remark by T.~Leinster quoted in~[\cite{Fr} p.~6]. 
\item
The derived $n$-ary operations also contain the standard $n-1$-simplex as label.
\end{enumerate}
\end{rem}
\begin{prop}
For every $A\in\mathbf{cAlg}_k$, 
$(A^{\N^*},\text{$+,\cdot$ point-wise}, +_q, q\in[0,1],\partial, f_n, n\in\N^*)$ is a partial convolution algebra which is furthermore a commutative unital ring for $(+,\cdot)$-point-wise.
\end{prop}

\begin{thm}
$(M_{\infty,\boxplus}(\R), \boxplus,\boxdot,\delta_0,\nu_{\infty,1,1}, \hat{\partial}^2,\mathbf{f}_n, n\in\N^*)$, $\mathbf{f}_n$ as in~(\ref{Frob_map}) and with $\hat{\partial}^2:=R^{-1}\circ\partial^2\circ R$, has the structure of an algebra over the set operad generated by $\Omega(0),\Omega(1)$ and $+,\cdot$. 
$(M_{\infty,\boxplus}(\R), \boxplus,\boxdot,+_q, q\in[0,1], \hat{\partial}^2,f_n,n\in\N)$ is an $(\Omega,E)$-algebra.
\end{thm}
\begin{rem}
\begin{itemize}
\item The notation $\hat{\partial}^2$ emphasises the fact that in order to have the décalage operator acting properly on measures we have to take its square, i.e. it is elliptic. 

\item By using e.g. the Berkovici-Pata-bijection or checking the properties directly, we can transfer this structure of a convolution algebra to the classical and boolean case, i.e. it applies also to the set $\operatorname{ID}(\R,*)$.
\end{itemize}
\end{rem}
\subsection*{Acknowledgements}
The author thanks the following people: John McKay for the discussions and his continuous interest, G. Cébron for the previous discussions and the helpful comments he made on a early version of the article and Roland Speicher for numerous inspiring discussions, his comments and questions and his support. The MPI in Bonn, he would like to thank for its hospitality.

\end{document}